\documentclass[12pt]{amsart}
\usepackage{amssymb,amsthm,amsmath,amsfonts}
\usepackage{graphicx}
\usepackage{epstopdf}
\usepackage{commath}
\usepackage{xcolor}
\usepackage{hyperref}
\usepackage[hypcap]{caption}
\def\tr{\;\mathrm{tr}\,}

\def\<{\langle}
\def\>{\rangle}

\def\E{{\mathbb E}}

\newcommand{\be}{\begin{equation}}
\newcommand{\ee}{\end{equation}}
      \newtheorem{theorem}{Theorem}[section]
       
       \newtheorem{corollary}[theorem]{Corollary}
       \newtheorem{lemma}[theorem]{Lemma}

\theoremstyle{remark}
\newtheorem{remark}[theorem]{Remark}

\title{Answer to a question by A. Mandarino, T. Linowski and K. \.{Z}yczkowski}

\author[M. Popa]{Mihai Popa}
\address[M. Popa]{Department of Mathematics\\ University of Texas at San Antonio\\ One UTSA Circle San Antonio\\ Texas 78249, USA\\
	\newline
	 and “Simon Stoilow” Institute of Mathematics of the Romanian Academy\\ P.O. Box 1-764\\ 014700 Bucharest, Romania}
\email{mihai.popa@utsa.edu}

\begin{document}
	\begin{abstract}
	A recent work by  A. Mandarino, T. Linowski and K. \.{Z}yczkowski
	left open the following question. If $ \mu_N $ is a certain permutation of entries of a 
	$ N^2 \times N^2 $ matrix  (``mixing map'') and 
	$ U_N $ is a $ N^2 \times N^2 $ Haar unitary random matrix, then is the family 
	$ U_N, U_N^{\mu_N}, ( U_N^2 )^{\mu_N}, \dots , ( U_N^m)^{\mu_N} $ asymptotically free? (here by 
	$A^{ \mu}$ we understand the matrix resulted by permuting the entries of $ A $ according to the permutation $ \mu $). This paper presents some techniques for approaching such problems. In particular, one easy consequence of the main result is that the question above has an affirmative answer. 
	\end{abstract}
	\maketitle
\section{Introduction}
In the recent years one can notice the emergence of a small but growing body of literature addressing questions connected to permutations of entries of various classes of matrices. For example, partial transposes (see part \ref{subsect:4:1} for a precise definition and more details) appear in connection to Quantum Information Theory in \cite{horod}, \cite{aubrun2}, \cite{fs}; furthermore, distributions of partial transposes of Wishart random matrices are described in \cite{aubrun}, \cite{banica-nechita}, also in \cite{wishart}, \cite{wishart2}; very interesting examples of entry permutations and other linear  transforms on matrix entries are studied in \cite{arizmendi}, \cite{collins-nechita}.
One should also mention the groundbreaking work \cite{adchawo} on random entry permutations on Haar unitary random matrices.

A recent paper in theoretical Physics, \cite{malizy}, authored by  A. Mandarino, T. Linowski and K. \.{Z}yczkowski, left open an interesting question regarding asymptotic free independence of certain random matrices. Briefly, the question is as follows. First, we naturally identify (quotient-remainder) each integer from 
$ 1 $
to 
$ N^2 $ 
to a pair of integers from 
$ 1 $ 
to
$ N $. 
This way, each entry 
$ (i, j) $ 
of a  
$ N^2 \times N^2 $
matrix is described by a $ 4 $-tuple 
$( a, b, c, d)$.
The map 
$ \mu_N $
is defined by permuting the entries via
$ (a, b, c, d) \mapsto (a, c, b, d)$.
We denote by 
$ A^{ \mu_N } $
the matrix obtained by permuting as above the entries of the matrix 
$ A  $. 
With these notations, \cite{malizy} asks if the family
$ U_{N^2} $,
$ U_{N^2}^{ \mu_N } $,
$ ( U_{N^2}^2)^{ \mu_N } $,
$\dots$,
$ ( U_{N^2}^P)^{ \mu_N } $
is asymptotically free,
where 
$ U_{N^2} $
is a $ N^2 \times N^2 $
Haar unitary random variable.

The techniques employed here to answer the question above track back to \cite{popa-gaussian}, where general entry permutations (including the map 
$ \mu_N $ from above) are studied in the framework of Gaussian random matrices. Ideas from \cite{popa-gaussian} 
were further developed in \cite{popa-hao-boolean}, \cite{popa-hao} for matrices with non-commutative entries, in the new and interesting work \cite{au} for Wiegner matrices, in \cite{mpsz1} for Haar unitaries. 
\cite{mpsz1} gives general necessary conditions, for a sequence of permutations 
$ \big(\sigma_N \big)_N $ 
such that
$ U_N^{ \sigma_N } $ 
is  asymptotically circular and free from 
$ \big( U_N \big)_N $;
furthermore, Consequence ... from \cite{mpsz1} shows that 
$ \big(U_{N^2}^{ \mu_N} \big)_N $
is asymptotically circular distributed and free from
$ \big( U_{N^2} \big)_N $, 
thus giving a partial answer to the question from 
\cite{malizy}.
Yet the step from 
$ \big(U_{N^2}^{ \mu_N} \big)_N $
to the family 
$ \big\{  \big( (U_{N^2}^ k )^{ \mu_N} \big)_N :\ k \in \mathbb{N}   \big\} $
seems non-trivial and requires the development of its own combinatorial machinery, as shown in the next pages.

Besides the Introduction, the paper is organized in three sections, as detailed below.

 Section 2, Background and Notations, gives a short review of the main results of \cite{co} and \cite{collins-sniady} regarding integration on the Unitary Group and lists some notation that will be widely used throughout the rest of the paper. The most important is equation (\ref{E:p:q}), which writes the expectation of the normalized trace of an arbitrary product of permuted Haar unitary random matrices and their adjoints as a sum over a family of pairings.
 
  Section 3, Preliminary Results, is  organized in two parts. First part gives a series of inequalities leading to Corollary \ref{cor:pq:2} which establishes that all summands in equation (\ref{E:p:q}) are asymptotically bounded and gives sufficient conditions for the asymptotic vanishing of such summands; second part further details the structure of the asymptotically non-vanishing summands.

Section 4, Main results, gives sufficient conditions for a family 
 $ ( \tau_{1, N})_N $, $ \dots $,
 $ ( \tau_{ P, N})_N $ 
 of sequences of entry permutations such that the family 
of correspondent permuted powers of Haar unitaries
$ \big\{ ( U_N^l )^{ \tau_{k, N}} : \  1 \leq l \leq M, \
1 \leq k \leq P \big\}$ 
is asymptotically free circular and free from 
$ ( U_N )_N $.

Section 5, Particular Cases, shows that the result from Section 4 can be easily applied to the entry permutation introduced in \cite{malizy} (giving an affirmative answer to the question left open there), but also to partial transposes or to a mix of the two.

\section{Background and notation}

\subsection{Review of the Unitary Weingarten Calculus}\label{weingarten}
One of the main results in \cite{collins-sniady} is the following. If 
$ U $ 
is a 
$ N \times N $
Haar unitary matrix, and 
$ u_{i, j} $ 
is its 
$ (i, j) $-entry, then
	\begin{align}\label{eqn:Weingarten}
	\E\left(
	u_{i_1,j_1}u_{i_2,j_2}
	\cdots u_{i_n,j_n}\overline{u_{i^\prime_1,j^\prime_1}}\, 
	\cdots\overline{u_{i^\prime_n,j^\prime_n}}\right)
	=
	\sum_{\sigma,\tau\in S_n}\Big(\prod_{k=1}^n\delta_{i_k,i^\prime_{\sigma(k)}}\delta_{j_k,j^\prime_{\tau(k)}}\Big) 
	\mathrm{Wg}_N (\sigma^{-1}\tau),
	\end{align}
where
$ \mathrm{Wg}_N $
is the unitary Weingarten function.
Furthermore, for
$\sigma\in S_n$ 
with cycle decomposition 
$\sigma=c_1\cdots c_{\#\sigma}$,
and
$ Cat_r $ 
denoting the $r $-th Catalan number, we have that
	\begin{align*}
		\mathrm{Wg}_N ( \sigma )=N^{-2n+\#\sigma}\prod_{i=1}^{\#\sigma} (-1)^{|C_i|-1} Cat_{|c_i|-1}+O(N^{-2n+\#\sigma-2}).
	\end{align*}

We will re-write the equation above using notations more suitable for the computations in the next sections. 
For 
$ m $ 
a positive integer, we shall denote by
$ [ m ] $ 
the ordered set 
$ \{ 1, 2, \dots, m \}$.
and denote by
$ P_2( m) $
the set of pair-partitions on 
$ [ m ]$.
Given a map 
$ \varepsilon : [ m ] \rightarrow \{ 1, \ast \} $,
we shall denote by 
$P_2^{\varepsilon} (m) $ 
the set (possibly void) of pair-partitions on
$ [ m ] $
that connect elements with different values of 
$\varepsilon$:
\[ P_2^{\varepsilon} (m)  = \{ p \in P_2(m):\ \varepsilon(k) \neq \varepsilon(p(k)) \textrm{ for all } k \in [ m ] \}. \]
As showed in \cite{mpsz1}, an equivalent form of equation \ref{eqn:Weingarten}, which is more suited for the computations in the present paper, is described as below:

\begin{align}\label{eqn:Wg:2}
\E \LARGE( 
u_{i_1,j_1}^{\epsilon(1)}
u_{i_2,j_2}^{\epsilon(2)}
\cdots u_{i_{m},j_{m}}^{\epsilon(m)}
\LARGE)
=
\sum_{p,q\in \mathcal{P}^\epsilon_{2}(m)}
\LARGE(
\prod_{k=1}^m\delta_{i_k,i_{p(k)}}\delta_{j_k,j_{q(k)}}
\LARGE)
\mathrm{Wg}_N( p , q ),
\end{align}
where the function
$ Wg_N ( \cdot, \cdot) $ 
is obtained from the unitary Weingarten function as follows. For
$ p, q \in P_2^\varepsilon ( m ) $, 
denote by
$ p \vee q $ 
the supremum of 
$ p $
and 
$ q $
in the lattice of the partitions of 
$ [ m] $ 
(see \cite{nica-speicher} for more details).
More precisely, if  
$ p \vee q $ 
has the block decomposition
$  B_1 B_2 \cdots B_{ | p \vee q |} $,
then 
each of the blocks 
has an even number of elements
$ \{ a_1, a_2, \dots, a_{2s}\}$
with
$ 	a_2=p(a_1),\,a_3=q(a_2),\,a_4=p(a_2),\ldots,a_1=q(a_{2s}) $.
Finally, we define 
 $ \mathrm{Wg}_N (p, q) = \mathrm{Wg}_N ( \sigma ) $.
  In particular, equation (\ref{eqn:Wg:2}) gives that
  \begin{align}\label{w-mult}
\mathrm{Wg}_N(p,q)=N^{-m+|p \vee q |}
\prod_{i=1}^{| p \vee q |} (-1)^{\frac{|B_i|}{2}-1} Cat_{\frac{|B_i|}{2}-1}+
O(N^{-2n+|p \vee q |-2}).
\end{align}

 \subsection{Notations}\label{notations}
 If 
 $ (i_1, i_2, \dots, i_m) $ 
 is a tuple indexed by the ordered set 
 $ [ m ] $
 and 
 $ S \subseteq [ m ] $
 is given by
 $ S = \{ a_1, a_2, \dots, a_s\} $ 
 with
 $ a_1 < a_2 < \dots < a_s $,
 we will use the notations
 $ (i_\nu)_{ \nu \in S} $
 for the $ s$ -tuple 
 $ ( i_{a_1}, i_{a_2}, \dots, i_{a_s})$.

  Throughout the paper, 
 $ U_N $ 
 will denote a
  $ N \times N $
 random Haar unitary matrix.
 If 
  $ \sigma $ 
 is a permutation of the set
    $ [ N ]^2 = \{ (i, j): 1 \leq i, j \leq N \} $,
then
$ U_N^{ \sigma} $ 
will denote the matrix defined by
$ \big[ U_N^{ \sigma} \big]_{ i, j} = \big[ U_N \big]_{ \sigma( i, j) } $.
Furthermore, if
 $ \vartheta \in \{ 1, \ast \} $, 
then we denote 
\begin{align*}
U^{( \sigma, \vartheta )} = \big( U^\sigma \big)^\vartheta=\begin{cases}
&U^{\sigma_s} \mbox{ if } \theta_s=1,\\
&\left(U^{\sigma_s}\right)^\ast \mbox{ if } \theta_s=\ast.
\end{cases}
\end{align*}

Suppose that 
$ n_1, n_2, \dots, n_r $ 
are positive integers,
$ \sigma_1, \sigma_2, \dots, \sigma_r $ 
are permutations on the set 
$ [ N ]^2 $
and 
$ \theta_1, \theta_2, \dots, \theta_r \in  \{ 1, \ast \} $.
The main goal of the following Section is to describe the leading term of the following expression (seen as a function in 
$ N $):
\begin{align*}
\E \circ \tr
\left(
\big[ (U^{n_1})^{\sigma_1}  \big]^{\theta_1}
\cdot
\big[ (U^{n_2})^{\sigma_2}  \big]^{\theta_2}
\cdots
\big[ (U^{n_r})^{\sigma_r}  \big]^{\theta_r}
\right) = 
\E\circ \tr\left(
( U^{ n_1} )^{(\sigma_1,\theta_1)}
\cdots
( U^{ n_r   } )^{(\sigma_{r},\theta_{r})}\right).
\end{align*}

We need to introduce more notations. First, put
\begin{align*}
a^{(\varepsilon)}=\begin{cases}
&a \mbox{ if } \varepsilon =1,\\
&\overline{a} \mbox{ if } \varepsilon =\ast\\
\end{cases}
\hspace{1cm} \textrm{  and } \hspace{1cm}
\varepsilon (i,j)=\begin{cases}
&(i,j) \mbox{ if } \varepsilon =1,\\
&(j,i) \mbox{ if } \varepsilon =\ast.\\
\end{cases}
\end{align*}
Therefore, for 
$ \pi_1, \pi_2 $ 
the canonical projections, 
\begin{align*}
\big[  
( U^{n})^{ ( \sigma, 1 ) }
\big]_{ i, j } = 
\big[ 
U^{ n  } ]_{ \sigma(i, j)}
=   \sum _{\alpha_1, \dots, \alpha_{n-1} \in [ N ]}
u_{ \pi_1 \circ \sigma (i, j), \alpha_1}
u_{ \alpha_1, \alpha_2}
\cdots
u_{ \alpha_{n-2}, \alpha_{n-1} }
u_{\alpha_{n-1}, \pi_2 \circ \sigma(i, j)}
\end{align*}
and
\begin{align*}
\big[
( U^{n } )^{( \sigma, \ast ) }
\big]_{ i, j}
=& 
\overline{ 
	\big[
	( U^{ n })^{ ( \sigma, 1 ) }
	\big]
}_{ j, i}\\
= &
\sum _{\alpha_1, \dots, \alpha_{n-1} \in [ N ]}
\overline {
	u_{ \pi_1 \circ \sigma (j, i), \alpha_1}
}
\cdot
\overline { u_{ \alpha_1, \alpha_2} }
\cdots
\overline { u_{ \alpha_{n-2}, \alpha_{ n -1} } }
\cdot
\overline {
	u_{\alpha_{ n -1}, \pi_2 \circ \sigma(j, i)}
}\\
=&
\sum _{\alpha_1, \dots, \alpha_{n-1} \in [ N ]}
\overline {
	u_{\alpha_{ n -1}, \pi_2 \circ \sigma(j, i)}}
\cdot
\overline { u_{ \alpha_{n-2}, \alpha_{ n -1} } }
\cdots
\overline { u_{ \alpha_1, \alpha_2} }
\cdot
\overline {
	u_{ \pi_1 \circ \sigma (j, i), \alpha_1}
}.
\end{align*}

Let
$ m ( 0 )= 0 $, 
and, for
 $ t =1, \dots, r $,
 let
$ m ( t ) =  m ( t -1) +  n (t )  $,
hence
$ m (r) = m $.
Define
$ \varepsilon: [ m ] \rightarrow \{ 1, \ast\} $
via
$ \varepsilon(s) = \varepsilon_s = \theta_t $
whenever
$ m( t - 1) + 1 \leq s \leq m(t) $.
For
$ t \in [ r]$,
denote by 
$ S_t $ 
the set
$ \{ m(t-1) + 1, m(t_1)+2, \dots, m(t) \} $.
Finally, for
$ i_1, i_2, \dots, i_m \in [ N ] $
with the identification
$ i_{ m + 1} = i_1 $,
define
$ \vartheta( i_1, \dots, i_m) = (k_s, l_s)_{1 \leq s \leq m}$, 
where
  \begin{enumerate}
	\item[$\bullet$] 
	if
	$ s \notin \{ m ( t ), m ( t ) +1:\ 0 \leq t \leq r \}$
	then
	$
	(k_s, l_s) = \epsilon_s ( i_s, i_{s+1} )
	$
	\item[$\bullet$] 
	if
	$ s \in \{ m(t)+1 : \ 0 \leq t \leq r-1 \}  $  
	then
	\begin{align*}
	(k_s, l_s) =
	\left\{
	\begin{array}{ll}
	( \pi_1 \circ 
	\sigma_t(i_{m ( t) +1}, i_{ m ( t + 1 ) +1}  ), i_{m ( t ) +2} )
	& \textrm{ if } \varepsilon_{s} = 1\\
	(i_{m ( t ) +2},
	\pi_2 \circ 
	\sigma_t(i_{m ( t) +1}, i_{m ( t+ 1) +1}  )    
	) &
	\textrm{ if } \varepsilon_s = \ast.
	\end{array}
	\right.
	\end{align*}
	\item[$\bullet$] 
	if
	$ s \in \{ m (t ) : 1 \leq t \leq r \}$,
	then
	\begin{align*}
	(k_s, l_s) =
	\left\{
	\begin{array}{ll}
	(i_{m ( t ) },
	\pi_2 \circ 
	\sigma_t(i_{m ( t -1) +1}, i_{m ( t ) +1}  )
	& \textrm{ if } \varepsilon_{s} = 1\\
	( \pi_1 \circ 
	\sigma_t(i_{m ( t-1) +1}, i_{m ( t ) +1}  ), i_{m ( t ) } )  
	) &
	\textrm{ if } \varepsilon_s = \ast.
	\end{array}
	\right.
	\end{align*}
\end{enumerate}
In particular,  the map
$ \vartheta $
is an injection from
$ [ N ]^m $
to 
$ [ N ]^{ 2m } $.

With these notations, equation (\ref{eqn:Wg:2}) gives that 
 \begin{align}\label{E:p:q}
\E\circ \tr \big(
( U^{ n_1} )^{(\sigma_1,\theta_1)}
\cdot 
( U^{ n_2 } )^{(\sigma_2,\theta_2)} 
\cdots &
( 
U^{ n_r   }
)^{(\sigma_{r},\theta_{r})}\big)\\
= &
\sum_{ i_1, i_2, \dots, i_m \in [ N ] }
\frac{1}{N}
\E 
\big(
u^{ ( \varepsilon_1)}_{ k_1, l_1}
\cdot
u^{ ( \varepsilon_2)}_{ k_2, l_2}
\cdots
u^{ ( \varepsilon_m)}_{ k_m, l_m}
\big) \nonumber\\
=& 
\sum_{p,q\in \mathcal{P}_2^\varepsilon(m)}
\big[
\frac{1}{N}
\mathrm{Wg}_N (p, q) 
\sum_{ i_1, i_2, \dots, i_m \in [ N ] }
\prod_{s=1}^m \delta_{k_s}^{k_{p(s)}}
\delta_{l_s}^{l_{q(s)}}
\big]\nonumber \\
=&
\sum_{p,q\in \mathcal{P}_2^\varepsilon(m)}
\mathcal{E}_N  ( p, q ) \nonumber
\end{align}  
where
\[
\mathcal{E}_N  ( p, q ) 
= \mathrm{Wg}_N( p, q) \cdot 
\frac{1}{N}
| \mathcal{A}_N^{(p, q)} |,
\] 
and
$ \mathcal{A}_N^{(p, q)} $
denotes the set
\begin{align*}
\big\{ \overrightarrow{i}=
(i_1, i_2, \dots, i_m) \in [ N ]^m :\ 
k_s = k_{p(s)} 
\textrm{ and }
l_s= l_{p(s)}
\textrm{ for all }
s = 1, 2, \dots, m 
\big\}.
\end{align*}

Finaly, for
$ S $
 a subset of
  $ [ m ] $, 
define
 \begin{align*}
\mathcal{A}_N^{(p, q)}(S)
& = \big\{
(k_s, l_s)_{ s \in  S} \in [ N ]^{ 2 | S | }: \ 
\textrm{ there exist }
(k_t, l_t )_{ t \in [ m ] \setminus S } \in [ N ]^{ 2 ( m - | S | )}
\textrm{ and }
\overrightarrow{i} \in \mathcal{A}_N^{ (p, q)}\\
 & \hspace{ 9 cm}\textrm{such that }
(k_s, l_s)_{ s\in [ m ]}
= \vartheta ( \overrightarrow{i} ) 
\big\}.
\end{align*}

\section{Preliminary results}

\subsection{Asymptotic boundedness of
	 $ \mathcal{E}_N (p, q)$ }

 \begin{lemma}\label{lemma:3:0}
 For  
 $ S, T \subset [ m ] $,
 denote 
 \[
 D ( S: T ) =
  \{ t \in T : \textrm{there exist some } a, b \in S\ \textrm{such that }  p(a) = q(b) = t  \}.
 \]
 
 If
 $ k \in [ r ] $
and the set
 $ S \subseteq [ m ] $ 
 are such that
 $ D ( S: S_k ) \neq \emptyset $,
 then
 \begin{align*}
 \big| \mathcal{A}_N^{( p, q)} 
 (S \cup S_k ) \big|
   \leq 
  \big| \mathcal{A}_N^{( p, q)} (S) \big|
  \cdot
   N^{ n(k) - | D  ( S: S_k ) |}.
 \end{align*}
 \end{lemma}
 \begin{proof}
 Suppose first that 
 $ n(k) = 1 $.
 Then 
 $ S_k $
  has only one element,
 $ m(k) = p(a)=q (b) $
 for some 
 $ a, b \in S $,
 so
 $ k_{ n(k)} = k_a $
 and
 $ l_{ n(k)} = l_b $,
 hence
 $ ( k_{m(k)}, l_{m(k)}) $
 is uniquely determined by
 $ ( k_s, l_s)_{ s \in S } $, 
 which gives
 $  \big| \mathcal{A}_N^{( p, q)} (S \cup S_k ) \big| = 
 \big| \mathcal{A}_N^{( p, q)} (S) \big| $.
  
   Suppose now that 
   $ n(k) > 1 $. 
   In this case, elements of 
   $ \mathcal{A}_N^{(p, q)}(S \cup S_k ) $
    are obtained by concatenated elements of 
    $ \mathcal{A}_N^{(p, q)}(S) $
    and tuples of the form
   $ (k_t , l_t )_{ m(k-1)+1 \leq t \leq m(k)} $ 
   which are uniquely determined by the
   $ ( n(k) + 1)$-tuple
   $ (i_{m(k-1)+1}, i_{m(k-1)+2}, \dots, i_{m(k) +1})$.
   
   If 
   $ D( S:S_k) \cap ( S_k \setminus \{ m(k-1) +1, m(k) \}) = \emptyset $,
   then 
   $ m(k-1) +1 \in D( S: S_k) $
   or
   $ m(k) \in D (S:S_k )$.
   
    If both 
    $ m(k-1) + 1 $ 
    and
    $ m(k) $
    are elements of 
    $ D( D:D_k) $, 
    then the $ 4$-tuple\\
     $ (k_{m(k-1)+1 }, l_{m(k-1) + 1}, k_{m(k)},
     l_{m(k)} ) $
     is uniquely determined by 
     $ ( k_s, l_s)_{ s \in S } $.
   Since 
   $ \sigma_k $ 
   is bijective, it follows that
    $ ( i_{ m(k-1)+1}, i_{m(k)+1} )$
    and
    $ i_{m(k-1)+2} $ and
    $ i_{m(k)} $
    are, in turn,
    uniquely determined by
    $ (k_{m(k-1)+1 }, l_{m(k-1) + 1}, k_{m(k)},
         l_{m(k)} ) $.
    So, for each
         $ ( k_s, l_s)_{s\in S } $,
         there are at most
         $ n(k)+1 - 3 $ possible choices (since 
      $ m(k-1) +2$  may equal $ m(k) $)
      for the tuple
       $ (i_{m(k-1)+1}, i_{m(k-1)+2}, \dots, i_{m(k) +1})$, 
  and the conclusion follows.
  
  If 
  $ m(k-1)+ 1 \notin D( S: S_k) $
  but
  $ m(k)  \in D( S:S_k) $, 
  then
   $ (k_{m(k)}, l_{m(k)} $ 
   are uniquely determined by
   $ (k_s, l_s)_{s\in S} $.
   That is
    $ i_{m(k)} $ 
    and one component of 
   $ \sigma_k\circ \theta_k (i_{ m(k-1) +1}, i_{m(k) +1}) ) $
   are uniquely determined by
   $ (k_s, l_s)_{s\in S} $.
   Therefore, there are at most
   $ N $
    possible choices for the couple 
   $ (i_{ m(k-1) +1}, i_{m(k) +1}) ) $
   and at most one choice for 
   $ i_{m(k)} $,
   and the conclusion follows.
    The case 
    $ m(k-1) +1 \in D( S:S_k) $
    and
    $ m(k) \notin D(S;S_k) $
    is similar.
    
    Finally, if 
    $ n(k) > 1 $
    and
  $ D( S:S_k) \cap ( S_k \setminus \{ m(k-1) +1, m(k) \}) \neq \emptyset $, 
  note first that whenever
  $ t 
  \in D( S: S_k) \cap ( S_k \setminus \{ m(k-1) +1, m(k) \}) $
  we have that
  $ i_{t} $
  and
  $ i_{t+1} $
  are uniquely determined by
  $ (k_s. l_s)_{s\in S} $.
  Indeed, 
  $ i_t = k_t = k_a $
  and
  $ i_{t+1} = l_t = l_b $
  for some 
  $ a, b \in S $ 
  with
  $ p(a) = q(b) = t $.
  Therefore, 
  at least
  $ | D(s;s_k) \cap (S \setminus \{ m(k-1) +1, m(k)\}) | + 1 $
  elements of the 
  $ (n(k)-1)$-tuple
   $ ( i_t )_{ t \in (S \setminus \{ m(k-1) +1, m(k) \})} $
   are uniquely determined by
   $ ( k_s, l_s)_{s\in S} $.
   Furthermore, as shown above, for each
   $ (k_s, l_s)_{s \in S} $
   for the couple
   $ ( i_{m(k-1)+1, m(k)+1}) $
   there are at most
   $ N $ 
   choices if only one of
   $ m(k-1)+1 $
   and 
   $ m(k) $
   are in 
   $ D( S:S_k) $,
  respectively one one choice
  if both are in $ D(S:S_k) $.
 \end{proof}
\begin{corollary}\label{cor:pq:3}
 Suppose that 
 $ S \subseteq [ m ] $
 and
 $ k \in [ r ] $
 are such that there exist some
 $ a \in S $ and
 some 
 $ t \in S_k $
 with
 $ p(a) = t $
 or 
 $ q(a) = t $.
 Then
  \begin{align*} 
    \big|     \mathcal{A}_N^{( p, q)} (S \cup S_k )  \big|
         \leq 
         N^{ n(k) - |  D (S: S_k )  | } 
         \cdot
   \big|      \mathcal{A}_N^{( p, q)} ( S  ) \big|.
         \end{align*}
\end{corollary}
\begin{proof}
If 
$ D(S:S_k) \neq \emptyset$,
then the result follows trivially from
Lemma \ref{lemma:3:0}. 
Suppose then that
$ D(S:S_k) = \emptyset$.

As seen before, the elements of
    $ \mathcal{A}_N^{(p, q)}(S \cup S_k ) $
     are obtained by concatenated elements of 
     $ \mathcal{A}_N^{(p, q)}(S) $
     and tuples of the form
    $ (k_t , l_t )_{ t \in S_k} $, 
    which are uniquely determined by the
    $ ( n(k) + 1)$-tuple
    $ (i_{m(k-1)+1}, i_{m(k-1)+2}, \dots, i_{m(k) +1})$.
  
 If 
 $ p(a) = b $ 
 for 
  $ a \in S $ 
  and 
  $ b \in S_k $, 
 then  
 $ k_b = k_a $ 
 is uniquely determined by the tuple
  $ ( k_s, l_s)_{s\in S } $.
  But
  $ k_b $ 
  is either an element of 
  $ \{ i_t : m(k-1) +1 < t < m(k)+1\} $
  or it is a component of
  $ \sigma_k \circ \theta_k 
  ( i_{ m( k-1) +1}, i_{m(k)} ) $.
  If 
  $ k_a = i_t $ 
  for some 
  $t $ 
  with
  $ m( k-1) +1 < t < m(k) + 1 $, 
  then for each 
  $ 2 | S | $-tuple 
  $ ( k_s, l_s)_{s \in S } $
  there are at most
   $ N^{ n(k)} $ 
   choices for the 
   $( n(k) +1)$-tuple
   $ (i_t)_{ m(k-1) +1 \leq t \leq m(k) + 1 } $.
   If
   $ k_a $
   is a component of 
   $ \sigma_k \circ \theta_k 
     ( i_{ m( k-1) +1}, i_{m(k)} ) $
     then, since
     $ \sigma_k $ 
     is bijective, for each
     $ ( k_s, l_s)_{s\in S } $
     there are at most
      $ N $
       choices for the couple
       $ i_{m(k-1) +1}, i_{m(k)+1} $
 hence at most
  $ n(k) $ 
  choices for the 
$( n(k) +1)$-tuple
   $ (i_t)_{ m(k-1) +1 \leq t \leq m(k) + 1 } $.
   
   The case 
   $ q(a) = b $
   is similar.
\end{proof}
The argument from the proof of Corollary \ref{cor:pq:3} above will further give the result below, that we will use in the next subsection.

\begin{lemma}\label{lemma:TQ}
Suppose that 
$ S \subset [ m ] $ 
and that 
$ k \in [ r ] $ 
is such that 
$ n (k ) > 1 $.
For each
$ \nu \leq n(k) -1 $, 
consider the sets
\begin{align*}
T_\nu &= \{ m(k-1) + 1, m(k-1) + 2, \dots, m(k-1) + \nu \} \\
Q_\nu &= \{m(k)-\nu +1, m(k)-\nu +2, \dots,  m(k)\}. 
\end{align*}
With this notations, we have that
\begin{enumerate}
\item[(i)] If
$ \theta_k =1 $ 
and
$ p(a) = m( k -1) + 1 $
for some
$ a \in S $,
or if
$ \theta_k = \ast $
and
$ q(a) = m(k-1) + 1 $ 
for some
$ a \in S $,
then
\[
\big| \mathcal{A}_N^{(p, a)} (S \cup T_\nu ) \big|\big| \leq 
N^{\nu - | D( S: T_{\nu})|}
\cdot
\big| \mathcal{A}_N^{(p, q)} (S) \big|.
\]
\item[(ii)] If
$ \theta_k = 1 $ 
and 
$ q(b) = m(k) $
for some
$ b \in S $,
or if
$ \theta_k = \ast $
and
$ p(b) = m(k) $
for some 
$ b \in S $, 
then
\[
\big| \mathcal{A}_N^{(p, q)} (S \cup Q_\nu) \big|
 \leq 
N^{ \nu - |D(S:Q_\nu)|}
\cdot
\big| \mathcal{A}_N^{(p, q)} (S) \big|
\]
\end{enumerate}
\end{lemma}

\begin{proof}
It suffices to show part (i); part (ii) follows then by taking adjoints.

As stated before, elements of 
$ \mathcal{A}_N^{(p, q)} (S \cup T_\nu) $
are concatenations of $ 2 | S | $-tuples from 
$\mathcal{A}_N^{(p, q)} (S) $ 
and
$ 2\nu $-tuples of the form 
$( k_t, l_t)_{ t \in T_\nu } $.

If
$ \theta_k =1 $, 
then
$ l_{t} = k_{t+1} $
for all
$ m(k-1)+1 \leq t \leq m(k)-1 $, 
so 
$(k_t, l_t)_{ t \in T_\nu} $
is uniquely determined by
$ k_{m(k-1)+1}$
and the
$ \nu $-tuple
$(l_t)_{t \in T_\nu }$.
Since
$ p(a) = m(k-1) +1 $, 
we have that
$ k_{ m(k-1) +1} = k_a $.
Furthermore, if
$ t \in D ( S:T_\nu)$, 
then
$ l_t = l_d $ for some 
$ d \in S $.
It follows that each tuple
 $(k_t, l_t)_{ t \in S \cup T_\nu}$
 from
 $ \mathcal{A}_N^{(p, q)}( S \cup T_\nu )$
 is uniquely determined by its components
 $ ( k_t, l_t)_{t\in S} $
 and
 $ (l_t )_{ t \in T_\nu \setminus D( S: T_\nu)} $, 
 hence the conclusion.
 
If
$ \theta_k =\ast$, the argument is similar.
Now
$ k_t = l_{t+1} $
for all
$ m(k-1)+1 \leq t \leq m(k)-1 $,
so
$(k_t, l_t)_{ t \in T_\nu} $
is now uniquely determined by
 $ l_{m(k-1)+1}$
and the
$ \nu $-tuple
$(k_t)_{t \in T_\nu }$. 
The condition
$ q(a) = m(k-1)+1 $
gives  that
$ l_a = l_{ m(k-1)+1 }$
and, as above, if
$ t \in D ( S:T_\nu)$, 
then
$ k_t = k_d $ for some 
$ d \in S $.
It follows that each tuple
 $(k_t, l_t)_{ t \in S \cup T_\nu}$
 from
 $ \mathcal{A}_N^{(p, q)}( S \cup T_\nu )$
 is uniquely determined by its components
 $ ( k_t, l_t)_{t\in S} $
 and
 $ (k_t )_{ t \in T_\nu \setminus D( S: T_\nu)} $, 
 hence the conclusion.
\end{proof}

 \begin{lemma}\label{lemma:3:1}
 For  
 $ S, T \subset [ m ] $,
 denote
 \[ 
 B (S:T ) = \{ B \textrm{ block in } p \vee q : B \subseteq S \cup T 
 \textrm{ and } B \nsubseteq S  \}. 
\]
 
 If 
 $ S_{k-1} \subseteq S $,
 or if
 $ S_{k+1} \subseteq S $
 then
  \begin{align*}
 \big| \mathcal{A}_N^{( p, q)} (S \cup S_k ) \big|
   \leq 
  \big| \mathcal{A}_N^{( p, q)} (S) \big|
   \cdot N^{ n(k) -
    | B ( S: S_k ) | }.
  \end{align*}
 \end{lemma}
 
 \begin{proof}
 
  It suffices to prove the result for the case
   $ S_{k-1} \subseteq S $;
   the case
    $ S_{k +1} \subseteq S $ 
    will follow by taking adjoints.
    
    Also, since 
       $ \varepsilon_t = \theta_k $
       for all 
        $ t \in S_k $,
        it follows that 
        $p $ and $ q $
        do not pair elements of 
        $ S_k $
        with other elements from 
        $ S_k $. 
        Therefore the set
        $D( S:S_k)  $ 
        has more elements than
         $ B (S : S_k ) $.
         It suffices then to show that
         \begin{align} \label{eq:D}
\big| \mathcal{A}_N^{( p, q)} (S \cup S_k )  \big|
         \leq 
         N^{ n(k) - |  D (S: S_k )  | } 
         \cdot
\big| \mathcal{A}_N^{( p, q)} ( S  ) \big| .
         \end{align}
         
    If 
    $ D( S:S_k) \neq \emptyset $
    then
    (\ref{eq:D}) 
    follows from Lemma \ref{lemma:3:0} above.
     Therefore suffices to show that if
     $ D(S:S_k) = \emptyset $
  then
   \begin{align}\label{simpl:1}
\big| \mathcal{A}_N^{( p, q)} (S \cup S_k ) \big|
 \leq 
N^{ n(k) } 
\cdot
\big| \mathcal{A}_N^{( p, q)} ( S  ) \big|.
   \end{align}      
    
   Fix 
   $ (k_s, l_s )_{s\in S } \in \mathcal{A}_N^{(p, q)}(S) $,
   i.e. all
   $ i_t $ 
   are uniquely determined for
    $ m ( k - 1) + 1 \leq t \leq m(k)+1 $.                                                                                                 
     The elements of 
     $ \mathcal{A}_N^{(p, q)}(S \cup S_k ) $
      are obtained by concatenated elements of 
      $ \mathcal{A}_N^{(p, q)}(S) $
      and tuples of the form
     $ (k_t , l_t )_{ t \in S_k} $, 
     which are uniquely determined by the
     $ ( n(k) + 1)$-tuple
     $ (i_{m(k-1)+1}, i_{m(k-1)+2}, \dots, i_{m(k) +1})$.
     But, as stated above, 
     $ i_{ m(k-1) + 1 } $ 
     is uniquely determined for each
      $ (k_s, l_s )_{s\in S } \in \mathcal{A}_N^{(p, q)}(S) $,
      so
      (\ref{simpl:1}) follows.
\end{proof}
\begin{corollary}\label{cor:pq:2}
${}$
\begin{enumerate}
\item[(i)]  $ \mathcal{E}_N (p, q) = O(N^0)= O(1) $
for any
 $ p, q \in P_2^{ \varepsilon}(m) $.
 \item[(ii)] With the notations from above, suppose that
 $ S = S_{ \alpha_1} \cup S_{\alpha_2 }
 \cup \dots \cup S_{\alpha_t}$
 for some
 $ \alpha_1, \dots, \alpha_t \in [ r ] $.
 If
   $ p, q \in P_2^{ \varepsilon}(m ) $
 are such that
 \begin{align}\label{eq:S:1}
 \big|\mathcal{A}_N^{ ( p, q)} (S) \big|
 = o \big( N^{ 1 + | S| 
 - | \{ B \textrm{ block in } p \vee q : \ B \subseteq S 
 \} | } \big)
\end{align}
then
$ 
\mathcal{E}_N ( p, q) = o(N^0)=o(1)$.
 
 \item[(iii)] If $ p \vee q $ has a block with more than one element in some 
 $ S_k $,
 then  \\
 $   \mathcal{E}_N (p, q) =  O(N^{-1})$.
\end{enumerate}
\end{corollary}

\begin{proof}
 Any element from 
 $ \mathcal{A}_N^{(p, q)} ( S_1 ) $
 is uniquely determined by 
 $ i_1, i_2, \dots, i_{m_1+1} $,
 hence
 \begin{align*}
 \big| \mathcal{A}_N^{(p, q)} ( S_1 )  \big|
 \leq N^{ m_1+1}.
 \end{align*}
  Applying Lemma \ref{lemma:3:1} $ r -1 $ times, we obtain that
  \begin{align*}
  |\mathcal{A}_N^{ (p, q)} | \leq N^{m_1+1} \cdot
  \prod_{t=2}^r N^{
  n_t  - | B ( \cup_{ \nu =1}^{ t-1} S_{ \nu} : S_t) |}
  = N^{m +1 - \sum_{t=2}^r | B( S_1 \cup \dots \cup S_{t-1} : S_t) | }.
  \end{align*}
  Since 
  $ p, q  \in P_2^{\varepsilon}(m )$
  no blocks of
   $ p \vee q $
   are contained in 
   $ S_1 $.
    It follows that 
     $ p \vee q $
   is the disjoint union of the sets
   $ 
    B( S_1 \cup \dots \cup S_{t-1} : S_t)  
  $
  for
  $ 2   \leq t \leq r-1   $,
  so 
   \begin{align*}
   | \mathcal{A}_N^{(p, q)} |  \leq 
   N^{m +1 - | p \vee q  | }.
   \end{align*} 
 
 Since 
 $ \textrm{Wg}_N(p,q)= O(N^{-m+|p \vee q |})$,
 we have that
 \begin{align*}
 \mathcal{E}_N  ( p, q ) 
 = \textrm{Wg}_N( p, q) \cdot 
 \frac{1}{N}
 | \mathcal{A}_N^{(p, q)} | \leq
  \textrm{Wg}_N( p, q) \cdot N^{m - | p \vee q  | }
  = O(1),
 \end{align*}
 so part (i) follows.
 
 To show part (ii), let 
 $ \beta = \min\{ t \in [r]: t \neq \alpha_\nu , 1\leq \nu \leq p \} $
 and
 $ S^\prime = S \cup S_\beta $.
 Lemma \ref{lemma:3:1} gives that
 \begin{align*}
\big|  \mathcal{ A}_N^{ ( p, q)} ( S^\prime ) \big|
 \leq 
\big|  \mathcal{ A }_N^{ ( p, q)} ( S ) \big|
  \cdot N^{ n(\beta) - | B ( S: S_{ \beta}) |  }
  = o \big(
  N^{1 + | S | + n( \beta ) - 
  | \{ B \in \pi_{p, q} :\ B \subseteq {S} \} |
   - | B ( S: S_{ \beta }) | } 
   \big)
 \end{align*}
 But
 $ | S^\prime | = | S | + | S_\beta | = | S | + n(\beta) $
 and the blocks of 
 $ p \vee q $
 contained in 
 $ S^\prime $
 are either contained in $ S $
 or elements of 
 $ B ( S: S_{\beta }) $.
 Hence
 $ S^\prime $ 
 satisfies equation (\ref{eq:S:1}).
 Iterating, it follows that 
 $ [ m ] $ satisfies (\ref{eq:S:1}), so
$ | \mathcal{A}_N^{(p, q)} |  =
o \big( 
   N^{1 + m  - | p \vee q | } \big) $, 
   therefore
\begin{align*}
\mathcal{E}_N ( p, q) = \textrm{Wg}_N( p, q) \cdot 
o \big( N^{m - | p \vee q |  } \big)
  = o(N^{0} ).
\end{align*}
 
 For part (iii), suppose that 
 $ B $ is a block of 
 $ p \vee q $
 such that
 $ S_k $ 
 contains more than one element of 
 $ B $.
 Via a circular permutation of the sets 
 $ S_1, S_2, \dots, S_r $, 
 we can suppose without loss of generality that
 $ k =r $.
 
 Denote 
 $ S = S_1 \cup \dots \cup S_{r-1} $.
  As seen above, Lemma \ref{lemma:3:1} gives that
  \begin{align*}
\big|  \mathcal{A}_N^{(p, q)}(S) \big| 
\leq 
  N^{ | S | +1 - \sum_{t=2}^{r-1}
  | B ( S_1 \cup \dots \cup S_{t-1}: S_t ) | },
   \end{align*}
and using again
    $ \displaystyle | p \vee q | =
    \sum_{t =2}^r 
    | B ( S_1 \cup \dots S_{t-1}: S_t ) |  $,
 the equation above reads
 \begin{align}\label{eq:B:r}
\big|  \mathcal{F}_N^{(p, q)}(S ) \big|
 \leq N^{ | S | + 1 - |  p \vee q | 
 + | B( S:  S_r ) |}.
 \end{align}
 
 On the other hand, since there is a block of 
 $ p \vee q  $
 with more than one element on 
 $ S_r $, 
 we have that the set 
 $ D ( S: S_r ) $
 has strictly more elements than
  $ B (S: S_r ) $, 
  so Lemma \ref{lemma:3:0} gives that
  relations 
  \begin{align*}
   | \mathcal{A}_N^{(p, q)} | 
   \leq 
  N^{1 + | S | + n(r) - | p \vee q |+ | B (S:S_r) |
  - | D ( S: S_r ) |  }
  \leq 
  N^{ m + 1 - | p \vee q |},
   \end{align*}
henceforth
  \begin{align*}
   \mathcal{E}_N(p, q) \leq \textrm{Wg}_N (p, q) \cdot 
   N^{m+1 - | p \vee q |-1 } = O(N^{-1}).
  \end{align*}
\end{proof}

\subsection{Structure of the asymptotically non-vanishing summands}

\begin{lemma}\label{lemma:structure}
Let $ p, q \in P_2^{\epsilon}(m) $.
Then either
$ \mathcal{E}_N (p, q) = O(N^{-1})$
or for each block 
$ B $ 
of 
$ p \vee q $,
with
$ B = \{ a_1, a_2, \dots, a_h \} $,
and
$ a_1 < a_2 < \dots < a_h $,
we have that  the following properties are satisfied by all
$ s \in [ h ] $,
with the identification
$ a_{ h + 1 }= a_1 $:
\begin{itemize}
\item[$\bullet$]$ \varepsilon_{ a_s } \neq \varepsilon_{ a_{ s + 1 } } $
\item[$\bullet$] $ \{ p(a_s), q(a_s) \} = \{ a_{  s - 1 }, a_{ s + 1 } \} $.
\end{itemize}
\end{lemma}
\begin{proof}
It suffices to show that  either
$ \mathcal{E}_N (p, q) = O( N^{-1} ) $
or every
 $ s \in [ h ] $ 
 satisfies the following property (again, we identify 
 $ a_{h +1} = a_1 $):
  \begin{align}\label{eq:c2:02}
   a_{s+1} \in \{ p(a_s), q(a_s)\}.
  \end{align}
  
For every
    $ t \in [ h ]$, 
define 
   $ b (t ) $
via 
   $ a_t \in S_{b(t) } $. 
From part (iii) of Corollary \ref{cor:pq:2}, we
can suppose that 
$ b(t_1) \neq b(t_2) $
   whenever
   $ t_1 \neq t_2 $ 
(otherwise 
$ \mathcal{E}_N (p, q) = O(N^{-1}) $). 

Suppose that (\ref{eq:c2:02}) is not satisfied for some 
$ s \in [ h ] $. 
 Via a circular permutation of the sets 
 $ S_1, S_2, \dots, S_r $
 we can further assume that
 $ b(s+1) = r $.
 Moreover, using part (ii) of Corollary \ref{cor:pq:2}, we can further assume that 
 $ S_r $
contains only one element, that is 
 $ a_h $,
of 
 $ B $,
so $ s = h -1 $.
 
To simplify the notions, let
$ b = b(h-1) $
and let 
   $ S = S_1 \cup S_2 \cup \dots \cup S_{b -1}$.
The condition
   $ a_1 < a_2 < \dots < a_p $
gives that
$ B \setminus \{ a_h \} \subseteq S $.
Next, since are under assumption that 
$ h -1 $ 
does not satisfy (\ref{eq:c2:02}), we have that
$ p(a_{h-1}), q( a_{h-1}) \in B \setminus \{ a_h\} $,
that is
$ a_{ h -1 } \in D( S: S_b ) $.
On the other hand, 
$ a_{h-1} \notin B( S: S_b) $,
because
$ a_h \in B (S \cup S_b ) $.
  Lemma 
   \ref{lemma:3:0} gives then
   \begin{align*}
\big|   \mathcal{A}_N^{ ( p, q)} (S \cup S_{b}) \big|
\leq
 \big|  \mathcal{A}_N^{ ( p, q)} (S ) \big|
   \cdot
    N^{  n( b ) - | D( S : S_{b} )|}
    < 
 \big| \mathcal{A}_N^{ ( p, q)} (S )  \big|
       \cdot
        N^{  n( b) - | B( S : S_{b} )|}.
   \end{align*}
   
Finally, as in the proof of Corollary \ref{cor:pq:2}, an application of Lemma \ref{lemma:3:1} gives
  \begin{align*}
   \mathcal{E}_N  ( p, q ) 
   = \textrm{Wg}_N( p, q) \cdot 
   \frac{1}{N}
   | \mathcal{A}_N^{(p, q)} | \leq
    \textrm{Wg}_N( p, q) \cdot N^{m - | p \vee q | -1 }
    = O(N^{-1}).
   \end{align*} 
\end{proof}
 \begin{lemma}\label{lemma:nc}
  Suppose that 
  $ p, q \in P_2^{\varepsilon}(m) $
  are such that there exist 
  $ a, b, c, d \in [ m ] $
with the following properties:
   \begin{enumerate}
   \item[$\bullet$] $ a < b < c < d $
   \item[$\bullet$] $ c \in \{ p(a), q(a)\} $ 
   and
   $ d \in \{ p(b), q(b ) \} $
   \item[$\bullet$] $ a \in S_{\alpha} $,
   $ b \in S_x $,
    $ c \in S_{ \beta } $
    and
    $ d \in S_y $
    with
    $ \alpha < x <  \beta < y $.
  \end{enumerate} 
  Then 
 $ 
 \mathcal{E}_N^{ ( p, q)} = O( N^{-1} ) $.
  \end{lemma}
\begin{proof}
Let
 $ S^\prime = S_1 \cup S_2 \cup \dots \cup S_{x-1} $.
 As in the proof of  part (i) of Corollary \ref{cor:pq:2}, we have that
 \begin{align*}
\big| \mathcal{A}_N^{ (p, q)} ( S^\prime ) \big|
  \leq
  N^{| S^\prime | +1 - | \{ B \textrm{ block in } p \vee q : B \subseteq S \} | }.
  \end{align*}
Henceforth, Corollary   \ref{cor:pq:3}
 gives that 
 \begin{align}\label{eq:S:4}
\big| \mathcal{A}_N^{( p, q)} ( S^\prime \cup S_{\beta}) \big|
 \leq
\big| \mathcal{A}_N^{(p, q)} (S^\prime ) \big|
  \cdot
  N^{ | S_ \beta| -| B (S^{\prime}: S_{ \beta}) | }
  \leq 
  N^{| S^\prime \cup S_{\beta} | + 1  
  - | \{ B \in p \vee q : \ 
  B \subseteq ( S^\prime \cup S_{\beta })\}}.
  \end{align} 
  
Denote
  $ S = S^\prime \cup S_{\beta} \cup S_{\beta-1}
  \cup \dots \cup S_{ x+ 1 } $.  
and note that
$ d \notin S $
from the assumption 
  $ \beta < y $.  
Equation (\ref{eq:S:4}) and Lemma \ref{lemma:3:1} give then  
  \begin{align}\label{eq:S:5}
 \big|  \mathcal{A}_N^{ ( p, q)}(S) \big|
  \leq
    N^{ | S | + 1 - | \{ B \in p \vee q : \ B \subseteq S \} | }
    \end{align}
    
  If 
  $ n(x) = 1 $
  then
  $ S_x $
  has a unique element
   $ b = m(x) $
   and, since 
   $ b $
   is in the same block with
   $ d \notin S $,
   we have that
   $ B (S:S_x ) = \emptyset $.  
On the other hand, each tuple
   $ ( k_s, l_s)_{ s \in S } $,
 uniquely determines the pair
 $ (  i_{ m(x) }, i_{m(x)+1 } ) $,
so it also  uniquely determines the pair
   $ (  k_{ m(x)}, l_{m(x)} ) = 
    \sigma_x \circ \theta_x  ( i_{ m(x)}, i_{ m(x) + 1} ) $. 
The inequality (\ref{eq:S:5}) gives then
    \begin{align*}
 \big|   \mathcal{A}_{N}^{ ( p, q)} ( S \cup S_x )\big|
    = 
 \big| \mathcal{A}_N^{(p, q)} (S)  \big|
    \leq
    N^{ | S \cup S_x | - | \{ B \in p \vee q : \ B \subseteq ( S \cup S_x ) \} | }
    \end{align*}
and the conclusion follows from part (ii) 
of Corollary \ref{cor:pq:2}.  

Next, suppose that
$ n(x) \geq 2 $, 
i.e.
$ S_x \setminus \{ b \} \neq \emptyset $.

If
$ b = m(x-1) + 1 $,
respectively if
$ b = m (x ) + 1 $
then, from the formulae describing the map
$ \vartheta $,
the couple
$ ( k_b, l_b ) $
is uniquely determined by the triples
$ ( i_{ m(x-1) + 1}, i_{ m (x-1) + 2}, i_{ m(x) + 1 } ) $
respectively    
$ ( i_{ m(x-1) + 1}, i_{ m (x) }, i_{ m(x) + 1 } )$.
Also, if
   $ m ( x -1) + 2 < b < m ( x ) + 1 $,
then
   $( k_b, l_b ) = \epsilon_b ( i_b, i_{b+1})$,
which is uniquely determined by
   $ (k_{b-1}, l_{b-1} )$ 
and 
   $ ( k_{b+1}, l_{b+1}) $.
Therefore
    $ ( k_b, l_b) $
 is uniquely determined by 
$ ( k_s, l_s)_{ s \in S\cup S_x\setminus \{ b \}}  $.   
So
     \begin{align*}
\big|     \mathcal{A}_N^{(p, q)} ( S \cup S_x ) \big|
= 
 \big| \mathcal{A}_N^{(p, q)} ( S \cup S_x \setminus \{ b \})\big|.
     \end{align*}
     
On the other hand, since
$ b $
is in the same block of 
     $ p \vee q $
 as 
     $ d \notin S $,
 we have that\\
  $ B ( S: S_x ) = B ( S : S_x \setminus \{ b \}) .$
According to Corollary \ref{cor:pq:2}(ii), it suffices to show that
\begin{align}\label{sx:1}
\big| \mathcal{A}_N^{(p, q)} (S \cup S_x \setminus \{b \} )
\big|
\leq
\big|\mathcal{A}_N^{(p, q)}(S) \big|
\cdot
N^{ n(x) - 1 - | D( S: S_x \setminus \{ b \} | }.
\end{align}
  
  As seen above, the couple
   $ ( i_{ m( x -1) + 1}, i_{m(x) + 1 }) $
   is uniquely determined by
   $ ( k_s, l_s)_{s \in S }$.
   So the tuple
   $ (k_t, l_t)_{ t \in S_x \setminus \{b\}} $
   is uniquely determined by 
   $ ( k_s, l_s)_{s\in S} $
   and by the $ ( n(x)-1)$-tuple
   $ ( i_\nu)_{ m(x-1)+1 < \nu < m(x)+1 }  $.
   
   If 
   $ t\in D (S:S_x \setminus \{ b \}) $
   then
   $ k_t = k_{t^\prime} $
   and
   $ l_t = l_{t^{\prime\prime}} $
   for some 
   $ t^\prime, t^{ \prime \prime} \in S $.
   Hence, if
    $ t = m(x-1) +1 $, 
    respectively 
    $ t = m( x) + 1 $,
    then
    $ i_{m(x-1) +2}$,
    respectively
    $ i_{m(x)}$
    are uniquely determined by
    $ ( k_s, l_s)_{s\in S}$.
    Also, if
    $ m(x-1) + 1 < t < m(x) + 1 $,
    then both
     $ i_{t} $
     and
     $ i_{t+1} $
     are uniquely determined by
   $ ( k_s, l_s)_{s\in S}$,
   so
   (\ref{sx:1}) follows.    
\end{proof}

\begin{lemma}\label{lemma:alpha:beta}
 Suppose that 
 $ \alpha , \beta \in [ r ] $
have the property that there exist some 
 $ a,b \in S_{\beta}$
 such that
 \begin{align}\label{pq}
     \theta_{ \alpha} \big(
     m(\alpha-1) + 1, m(\alpha) \big)
     =
     \big( p(a), q(b) \big) .
\end{align}
   Then either
   $ \mathcal{E}_N (p, q) = O(N^{-1} )$ 
   or
   $ p(S_{\alpha} ) \subseteq S_{\beta} $ 
   and
   $ q ( S_{\alpha}) \subseteq S_{\beta} $.
\end{lemma}

\begin{proof}
Via a circular permutations of the sets 
   $ S_1$, $ S_2, \dots, S_r $,
 we can suppose that
   $ \beta < \alpha $.
   
If 
  $ n( \alpha ) =  1 $,
then 
  $ S_\alpha $ 
has only one element, 
  $ m(\alpha)= m(\alpha -1) + 1 $,
so the result is trivial.

Suppose that 
  $ n(\alpha ) \geq 2 $ 
and that there exists some
  $ \nu \in S_{\alpha} $ 
such that
  $ p(\nu) \notin S_{\beta} $
or
  $ q(\nu) \notin S_{\beta}$.
Let us denote by 
  $ s $ 
the smallest element of 
  $ S_{\alpha } $ 
with this property.

First, note that, by construction,  
$ s \notin D(S:S_{\alpha}) $,
so
$ D(S:S_\alpha) = D(S:S_\alpha \setminus \{ s \}) $.

Next, note that
\begin{align}
\label{eq:s:1}
\big| \mathcal{A}_N^{(p, q)} (S \cup S_\alpha) \big|
 = 
 \big| \mathcal{A}_N^{(p, q)} (S \cup
( S_\alpha \setminus \{ s \}))  \big|
\end{align}
 Indeed, if 
 $ s = m(\alpha -1) + 1 $ 
 and
 $ \theta_{ \alpha} = 1 $, 
 then
 $ l_s = k_{s+1} $
 while
 (\ref{pq}) gives 
 $ p(a) = s $ 
 so 
 $ k_a = k_s $;
if 
$ s = m(\alpha -1) + 1 $ 
 and
 $ \theta_{ \alpha} = \ast $,
 then
 $ k_s = l_{s+1} $
 while
 (\ref{pq}) gives
 $ q(b)= l_s $.
 In both cases, it follows that any tuple
 $ (k_t, l_t)_{t \in S \cup S_{\alpha} } $
 is uniquely determined by its components
 $ (k_t, l_t)_{ t \in S \cup ( S_\alpha \setminus \{s\})}$,
 so (\ref{eq:s:1}) holds true. A similar argument works for the case
 $ s = m(\alpha) $.
 Finally, if 
 $ m(\alpha-1) + 1 < s < m(\alpha)$,
 then
 $ (k_s, l_s) = ( l_{s-1}, k_{s+1}) $ 
 if
 $ \theta_{ \alpha}= 1 $, 
 respectively
 $ (k_s, l_s) = ( l_{s+1}, k_{s-1}) $ 
 if
 $ \theta_{ \alpha} = \ast$,
 so (\ref{eq:s:1}) holds true.
 
 Let 
 $ A_1 $
 be the possibly empty set
 $ A_1 = \{ m(\alpha-1) + \nu :\ 0 \leq \nu < s - m( \alpha -1 )  \} $
 and denote
 $ A_2 = S_\alpha \setminus ( \{ s \} \cup A_1 ) $. 
 Applying Lemma \ref{lemma:TQ} we then obtain
 that:
 \begin{align*}
\big| \mathcal{A}_N^{(p, q)} (S \cup A_1 \cup A_2) \big|
& \leq 
 N^{|A_2| - | D(S\cup A_1 : A_2 )| } 
\cdot
\big| \mathcal{A }_N^{(p, q)} (S \cup A_1 )  \big|   \\
& \leq 
N^{|A_2| - | D(S\cup A_1 : A_2 )| } 
\cdot
N^{|A_1| - | D(S: A_1 )| } 
\cdot
\big| \mathcal{A}_N^{(p, q)} (S ) \big|  .
\end{align*}   
   Since 
   $ A_1 $ 
   and
   $ A _2 $
   are disjoint, 
   $ D( S: A_1 \cup A_2 ) = D(S: A_1) \cup D(S:A_2) $.
   Moreover, Lemma \ref{lemma:structure} gives that either
   $ \mathcal{E}_N{(p, q)} = O(N^{-1})$
   or
$ D(S:A_1) \cap D(S:A_2) = \emptyset$
and   
   $ D(S\cup A_1: A_2 ) = D (S : A_2) $.
   Hence, we obtain that either
$ \mathcal{E}_N{(p, q)} = O(N^{-1})$
or
\begin{align*}
\big| \mathcal{A}_N^{(p, q)} (S \cup A_1 \cup A_2) \big|
 \leq
 N^{|A_1| + | A_2| - | D(S:A_1 \cup A_2 )| }
 \cdot
\big|  \mathcal{A}_N^{(p, q)} (S) \big|  .
\end{align*}   
Since 
$ S_\alpha $
 is the disjoint union of
$ A_1 $ 
and 
$ A_2$, 
the equality
(\ref{eq:s:1}) gives that the relation above is equivalent to
\begin{align*}
\big| \mathcal{A}_N^{(p, q)} (S \cup S_\alpha) \big|
 \leq
N^{ n(\alpha)-1-|D(S:S_{\alpha})| } 
\cdot
\big| \mathcal{A}_N^{(p, q)} (S) \big|
\end{align*}
which, from Corollary \ref{cor:pq:3} also implies that
$ \mathcal{E}_N{(p, q)} = O(N^{-1})$.
\end{proof}

\begin{lemma}\label{lemma:1-1}
Suppose that 
$\alpha, \beta \in [ r] $
are such that
\begin{align}
\label{pqab}
p( S_{\alpha} ) = S_\beta = q ( S_{ \alpha}).
\end{align}
Then either 
$ \mathcal{E}_N^{(p, q)} = O(N^{-1})$
or for each
$ \nu \in [ n ( \alpha)] $,
we have that 
\[
 p ( m( \alpha -1) + \nu ) 
= q ( m ( \alpha -1) + \nu )= m( \beta ) + 1 - \nu. 
\]
\end{lemma}

\begin{proof}

First, note that  condition (\ref{pqab}) implies 
$ \theta_{ \alpha}\neq \theta_{ \beta} $;
without loss of generality, we can suppose that
$ \theta_{ \alpha} = 1 $
and
$\theta_{ \beta} = \ast$.

Next, note that, since
 $ p $ and $ q $
 are injective, the condition
  (\ref{pqab}) 
  implies 
   $ n(\alpha ) = n( \beta ) $.
Also, (\ref{pqab}) gives that 
$ S_{\alpha} \cup S_{\beta} $
 is a union of blocks of 
$ p \vee q $.
 Corollary \ref{cor:pq:2} gives that either
 $ \displaystyle 
  \mathcal{E}_N ( p, q) = O(N^{-1}) $,
or each of these blocks have exactly two elements, one in 
$ S_{\alpha}$
and one in 
$ S_{\beta}$.
The last property means that
$ p(t) = q(t) $
for all
$ t \in S_{\alpha } \cup S_{\beta}$
and that
\[
 | B( S_{\alpha}: S_{\beta}) | = 
|  D( S_{\alpha} : S_{\beta})| 
= n(\alpha).
\]
Therefore Lemma \ref{lemma:3:1} gives 
\begin{align*}
\big| \mathcal{A}_{N}^{(p, q)} ( S_{\alpha} \cup S_{\beta}) \big|
 \leq 
\big| \mathcal{A}_{N}^{(p, q)} ( S_{\alpha} ) \big|
  \cdot
  N^{ n( \beta) - n( \alpha)}
  = 
\big|  \mathcal{A}_{N}^{(p, q)} ( S_{\alpha} ) \big|.
 \end{align*}
and, utilizing Corollary \ref{cor:pq:2}, we obtain that either
$ \mathcal{E}_N{(p, q)} = O(N^{-1}) $ or
$
\big| \mathcal{A}_{N}^{(p, q)} ( S_{\alpha} ) \big|
 = N^{ n(\alpha) +1}.
$

As seen before, each 
$(2n(\alpha))$-tuple
$ ( k_s, l_s)_{s\in S_\alpha } $
is uniquely determined by the
$ (n(\alpha)+1) $-tuple 
$ \big( i_{m(\alpha-1)+1}, i_{m(\alpha-1) + 2}, \dots, i_{m(\alpha) +1} \big) $.
Moreover, using that 
$ \theta_{ \alpha} = 1 $, 
 we have that the following relations are satisfied whenever
  $ { m(\alpha-1) +2} \leq t \leq m(\alpha)$:
  \begin{align*}
  &i_{t} = k_{t +1} =l_t \\
  &\sigma_{ \alpha } ( i_{ m(\alpha-1)+1}, i_{m(\alpha) +1 }) = ( k_{ m( \alpha-1)+1}, l_{m(\alpha)}).
  \end{align*}
 Hence the $ 2n(\alpha)$-tuple
 $ ( k_s, l_s)_{s\in S_{\alpha}} $
 is uniquely determined by  $ l_{m(\alpha)} $
 and the
 $( n(\alpha))$-tuple
 $ (k_{m(\alpha-1) +1}, k_{m(\alpha-1)+2}, \dots, 
 k_{m(\alpha)}) $.
 
 On the other hand, since 
 $ \theta_{ \beta} = \ast$, 
 we have
 \[ \sigma_{ \beta} (i_{ m(\beta-1) + 1}, i_{ m(\beta)} ) =
  ( k_{ m(\beta)}, l_{ m( \beta -1) +1})
 \]
 and
 $ i_t = k_t = l_{ t +1} $
 for
 $ m( \beta -1) + 2 \leq t \leq m( \beta ) $.
 
 Next, remark that either 
 $\mathcal{E}_N (p, q)= O(N^{-1})$ 
 or 
 $ q ( m(\alpha)) = m( \beta -1) + 1 $.
 Indeed,  if 
 $ q ( m(\alpha)) \neq m( \beta - 1) + 1 $
 then
  $ q ( m ( \alpha)) -1 \in S_{\beta} $.
  Since 
  $ l_t = l_{ q(t) } $,
  we get
 $ l_{ m( \alpha)} = l_{ q ( m( \alpha))} $,
 and, using that
 $ q ( m( \alpha )) \in S_{\beta} $,
 it follows that
 $
  l_{ q ( m ( \alpha) ) } = k_{ q ( m ( \alpha) ) -1} $
  which gives
  $ l_{ m( \alpha)} = k_{ p ( q ( m ( \alpha) -1))}$.
  This implies that 
  $ ( k_s, l_s)_{ s\in S} $
  is uniquely determined by the
  $ n(\alpha)$-tuple
  $ ( k_t)_{ t \in S_{\alpha}}$, 
  i.e.
  $\big| \mathcal{A}_N^{(p, q)}(S_{\alpha}\cup S_{\beta})\big|
  \leq N^{ n(\alpha)} $,
  therefore
   $ \displaystyle \mathcal{E}_N ( p, q) = O(N^{-1})$.
   
   For
  $ m( \alpha -1) + 1 \leq t \leq m(\alpha) - 1 $
  we have that 
  $ t + 1 \in m( \alpha) $, 
  hence
  $ k_{t+1} = l_ t $.
 But
  $ l_t = l_{q(t)} = k_{ q(t) -1} $,
  since
   $ q(t) \in S_{\beta} $.
   On the other hand, 
   $ k_{ q (t) -1} = k_{p( q(t) -1)} $,
   which gives
   $ k_{t+1}= k_{p( q(t) -1 ) } $
   with both
   $ t+1 $ and  $ p( q(t) -1 ) $ elements of 
   $ S_{\alpha} $.
   If 
    $ t +1 \neq p ( q(t) -1 ) $, 
    then there are at most 
    $ N^{ n( \alpha)} $ 
    possible choices for the tuple
   $ (k_{m(\alpha-1) +1}, k_{m(\alpha-1)+2}, \dots, 
    k_{m(\alpha)}, l_{m(\alpha)} )  $,
    therefore, according to Corollary \ref{cor:pq:2}, we have that
    $  \mathcal{E}_N ( p, q) = O(N^{-1}) $.
    It follows that
     $ t + 1 = p( q(t) - 1) $, 
     hence, since $ p $ and  $ q $ coincide on 
     $ S_{\alpha} \cup S_{\beta } $,
     we have that 
     $ q(t+ 1) = q(t) -1 $,
     and an inductive argument gives the conclusion.
    \end{proof}

\section{Main Results}

\begin{lemma}\label{lemma:sup}
	Suppose that there exists some 
	$ k \in [ r ] $
	such that the sequence
	$ \big( \sigma_{k, N} \big)_N $  
	satisfies the following property:
	\begin{equation}\label{max}
	\lim_{N \rightarrow \infty} \frac{1}{N}
	\sup_{(a,b) \in [ N ]^2}
	\sum_{ \alpha =1, 2 }
	 | \big\{ \nu \in [ N ] :\  a \in \{ \pi_\alpha\circ \sigma (\nu, j),  \pi_\alpha \circ \sigma ( j, \nu) \} \big\}| = 0,
	\end{equation}
where 
$ \pi_1$, $ \pi_2 $ are the canonical projections.
	
	If
	$ p. q \in P_2^{\varepsilon}(m) $
	are such that
	$ \displaystyle \limsup_{N \rightarrow \infty }
	\mathcal{E}_{N}^{( p, q)}  \neq 0 $
	and 
	$ a, b \in [ m ] $ 
	are such that
	\[
	\big( p(a), q(b) \big) = \theta_k \big( m( k-1 ) + 1, m( k ) \big),
	\]
	then 
	$ a $ and $ b $ are both elements of the same 
	$ S_\alpha $ 
	for some 
	$ \alpha \in [ r ] $.
\end{lemma}

\begin{proof}
 First assume that 
$ \theta_k =1 $;
let us suppose that
$ m(s-1)+1 = p(a) $
and
$ m(s) = q(b) $
for some
$ a \in S_\alpha $
and 
$ b \in  S_\beta $
such that
$ \alpha \neq \beta $.
Via a circular permutation of the sets
$ S_1, S_2, \dots, S_r $,
we can further suppose that either
$ \alpha < k < \beta $
or
$ \beta < k < \alpha $.

Denote
$ S = S_1 \cup S_2 \cup \dots \cup S_{ k -1} $.
Since
$ \alpha \neq k \neq \beta $,
we have that
$ S \neq \emptyset $.
From Corollary \ref{cor:pq:2}, if suffices to show that

\begin{align}\label{eq:s:s_k}
\big| \mathcal{A}_{N}^{( p, q)} \big( S \cup S_k \big) \big|
= 
\big| \mathcal{A}_{N}^{( p, q)} \big( S  \big) \big|
\cdot
o\big(   N^{ n (k) - | B ( S: S_k )| }  \big). 
\end{align}

 Suppose first that 
$ \alpha < k < \beta $. 
As proved in Lemma \ref{lemma:TQ}, 
we have that
\begin{align}\label{eq:s:l2}
\big| \mathcal{A}_{N}^{ ( p, q)}
( S \cup ( S_k \setminus \{ m(k ) \} ) ) \big|
\leq
\big| \mathcal{A}_{N}^{( p, q)} \big( S  \big) \big|
\cdot
N^{ n(k) - 1 - | D( S: S_k \setminus \{ m( k ) \} )|}.
\end{align}

Since 
$ q ( m (k) ) = b \notin S $,
we have that 
$ m(k) \notin D ( S: S_k ) $,
therefore  
$ D(S: S_k \setminus \{ m(k)\} ) = D( S: S_k ) $.
So (\ref{eq:s:l2}) gives
\begin{align*}
\big|\mathcal{A}_{N}^{ ( p, q)} 
\big( S \cup ( S_k \setminus \{ m(k ) \} ) \big) \big|
\leq
 \big|  \mathcal{A}_{N}^{( p, q)} \big( S \big) \big|
\cdot
N^{ n(k) - 1 - | D( S: S_k )|}
\leq
\big| \mathcal{A}_{N}^{( p, q)} \big( S \big)  \big|
\cdot
N^{ n(k) - 1 - | B( S: S_k )|},
\end{align*}
that is, to show (\ref{eq:s:s_k}), if suffices to prove that
\begin{align*}
\big| \mathcal{A }_{N}^{ ( p, q)} 
\big( S \cup  S_k  \big)  \big|
=
o(N ) \cdot 
\big| \mathcal{A}_{N}^{ ( p, q)}
\big(  S \cup ( S_k \setminus \{ m(k ) \} )  \big) \big|,  
\end{align*}
i.e. it suffices to show that for any given tuple 
$( k_t, l_t)_{ t \in S \cup S_k \setminus \{ m(k)\}} $
there are 
$ o(N) $
pairs
$ ( k_{ m(k)}, l_{m(k)} ) $ 
such that their juxtaposition
$ ( k_t, l_t)_{ t \in S \cup S_k } $
is in
$ \mathcal{A}_N^{(p, q} ( S \cup S_k )$.
By construction (see the notations from \ref{notations}), 
we have that 
$ k_{ m (k) } = l_{ m(k ) -1 } $,
therefore
$ k_{ m(k)} $
is uniquely determined by the tuple
$( k_t, l_t)_{ t \in S \cup S_k \setminus \{ m(k)\}} $.
Next, according to Notations \ref{notations}, we have that
\begin{align*}
i_{ m(k-1) + 1} = \left\{
\begin{array}{ll}
\pi_1  \circ \sigma_{k-1, N} \big(k_{ m(k-2) +1}, l_{ m(k-1)}   \big) \textrm{ if } 
\theta_k =1 \\
\pi_1 \circ \sigma_{k-1, N} \big(l_{ m(k-2) +1}, k_{ m(k-1)}   \big) \textrm{ if } 
\theta_k = \ast
\end{array}
\right.
\end{align*}
and
\begin{align*}
 l_{ m(k)}  = \pi_2 \circ \sigma_{k, N} ( i_{m(k-1)+1}
, i_{ m(k) +1 } ).
\end{align*}
 Therefore
 $ i_{ m(k-1) + 1} $
also  is uniquely determined by 
$( k_t, l_t)_{ t \in S \cup S_k \setminus \{ m(k)\}} $,
and, given the tuple
$( k_t, l_t)_{ t \in S \cup S_k \setminus \{ m(k)\}} $,
we have that 
$ l_{m(k)} $ 
is uniquely determined by 
$ i_{ m(k) + 1} $.

But, using property (\ref{max}) for the second relation, we have that
\begin{align*}
| \big\{ i_{ m(k) + 1}: \ ( k_t, l_t)_{ t \in S \cup S_k } \in
\mathcal{A}_N^{(p, q} ( S \cup S_k ) \big\}| \leq 
| \big\{ 
\nu :\ k_{m(k-1) + 1} = \pi_1 
\circ \sigma (i_{ m(k-1) +1}, \nu )
\big\}| = o(N),
\end{align*}
and the conclusion follows.

The case 
$ \beta < k < \alpha $ 
is similar: this time let
$ S = S_{k +1} \cup \dots \cup S_r $ 
and, via Lemma \ref{lemma:TQ}, it suffices to show that 
\begin{align*}
\big| \mathcal{A}_{N}^{ ( p, q)}
\big( S \cup  S_k  \big) \big|
=
o(N ) \cdot 
\big| \mathcal{A}_{N}^{ ( p, q)}
\big(  S \cup ( S_k \setminus \{ m(k-1 )+1 \} )  \big)
\big| . 
\end{align*}
 Given a tuple 
 $( k_t, l_t)_{ t \in S \cup S_k \setminus \{ m(k-1)+1\}} $,
 it uniquely determines, from the constructions in Notations \ref{notations}, both 
 $ l_{m(k-1) +1} $ 
 and
 $ i_{ m(k)} $; 
 also, since 
 $ p(a) = m(k-1)+1 $,
 that is 
 $ k_a = k_{ m(k-1) +1}$, 
 for some 
 $ a \in S_\alpha \subseteq S $,
 the tuple also uniquely determines 
 $k_{ m(k-1) +1} $.
 So the pair 
 $ ( k_{m(k-1) + 1}, l_{m(k-1) + 1})$
 is uniquely determined by the tuple 
 $( k_t, l_t)_{ t \in S \cup S_k \setminus \{ m(k-1)+1\}} $
 and by
 $ i_{ m(k-1) + 1} $.
 Furthermore, the number of possible choices for 
 $ i_{ m(k-1) +1} $ 
 is bounded above by the cardinality of the set
 \begin{align*}
 \{ \nu :\  k_a = \pi_1 \circ 
 \sigma_{ k, N} ( \nu, i_{ m(k)})    \}
 \end{align*}
 which, from Property (\ref{max}), equals $o ( N ) $.
 
 Finally, the case the case 
 $ \theta_k = \ast $
 follow similarly, by taking adjoints, switching 
 $ p $ and $ q $
  and using Property (\ref{max}) for 
 $ \pi_2 $.
\end{proof}

An immediate consequence of Lemmata \ref{lemma:sup} and \ref{lemma:alpha:beta} is the following.

\begin{corollary}\label{cor:sup}
	Suppose that $ \alpha \in [ r]  $ is such that the sequence 
$ \big( \sigma_{ \alpha, N} \big)_N $ 
satisfies property 
(\ref{max})
and that the pairings
$ p, q \in P_{2}^{\varepsilon} ( m ) $ 
are such that 	
	$ \displaystyle \limsup_{N \rightarrow \infty }
\mathcal{E}_{N}^{( p, q)}  \neq 0 $.
Then there exists some
$ \beta \in [ r ] $ 
 such that 
 $ p(S_\alpha), q(S_\alpha) \subseteq S_\beta $.
\end{corollary}

 The main result of the paper is the theorem below.
 
   \begin{theorem}\label{main}
 Let  
    $ P $ 
 be a positive integer.
 For each $ k = 1, 2, \dots, P $,
 suppose that 
 $ \big( \tau_{ k, N } \big)_N $ 
 is a family of permutations that satisfies condition (\ref{max}) with each
 $ \tau_{ k, N } $ 
  a permutation on 
  $ [ N ] \times [ N ] $.  
 If for each distinct 
 $ a, b \in \{ 1, \dots, P \} $ 
 the sequences  
    $ \big(\tau_{ a, N} \big)_N $
    and 
   $ \big( \tau_{b, N } \big)_N $
   satify the condition
\begin{align}\label{2:free}
\lim_{N \rightarrow \infty}\frac{1}{N^2}
\big|  \big\{ (i_1, i_2, i_3) \in [ N ]^3 : \  \tau_{a, N}(i_1, i_2) \in \{  \tau_{ b, N} (i_1, i_3) , \tau_{b, N} (i_3, i_2) \}
\big\}
\big| = 0
\end{align}   
   then, for any positive integer 
   $ M $, the family 
 $ \big\{ ( U_N^l )^{ \tau_{k, N}} : \  1 \leq l \leq M, \
 	1 \leq k \leq P \big\}$
 is asymptotically (as 
 $ N \rightarrow \infty $
 ) free circular of variance $ 1 $
 and free from   
 $ U_N $.
 \end{theorem}

 As in Section \ref{notations}, let 
$ n_1, \dots, n_r $
be positive integers,
for 
$ k =1, \dots, r $ 
let
$ \big(\sigma_{k, N})_N $ 
be sequences of permutations either equal to 
$\big (\textrm{Id}_N \big)_N $
or from the set
$ \big\{  \big( \tau_{k, N}\big)_N :\ 1 \leq P \big\} $
and let
$ \theta_1, \theta_2, \dots, \theta_r \in  \{ 1, \ast \} $.
We also assume that 
$ n_k =1 $
whenever
$ \big( \sigma_{ k, N } \big)_N = \big( \textrm{Id}_N \big)_N $.
 Furthermore, as in Section \ref{notations}, let 
 $ m(0) = 0 $ 
 and, for
 $ 1 \leq t \leq r $, 
 put
 $ m(t) = m(t -1) + n_t $;
 denote
 $ m = m(r) $, 
 then define
 $ \varepsilon: [ m ] \rightarrow \{ 1, \ast\} $
 via
 $ \varepsilon(s) = \varepsilon_s = \theta_t $
 whenever
 $ m( t - 1) + 1 \leq s \leq m(t) $.

We shall prove the theorem by applying the results from the previous sections to the moment-free cumulant expansion of 
\begin{align*}
\E \circ \tr
\left(
\big[ (U_N^{n_1})^{\sigma_{1, N}}  \big]^{\theta_1}
\cdot
\big[ (U_N^{n_2})^{\sigma_{2, N}}  \big]^{\theta_2}
\cdots
\big[ (U_N^{n_r})^{\sigma_{r, N}}  \big]^{\theta_r}
\right).
\end{align*}
and utilizing the following Lemma.

\begin{lemma}\label{lemma:aux}
	Under the assumptions of Theorem \ref{main}, suppose that 
$ p, q \in P_2^{\varepsilon}(m) $
are such that
$\displaystyle \lim_{N \rightarrow \infty}
 \mathcal{E}_N(p, q) \neq 0 $. 
Further suppose that 
$ x, y \in [ r ] $
	are such that 
$ x < y $ 
	and, for 
$ a = m(x) $,
$ b = m(y) $,
 the set 
$ S = [ b ] \setminus [ a ] $
	have the property that
$ p(S) = q(S) = S $. 
Then:
\begin{itemize}
	\item[(i)] If 
	$ v \in [ r ] $
	is such that 
	$ \big( \sigma_{v, N} \big)_N 
	\neq \big( \textrm{Id}_N \big)_N $, 
	then there exists some 
	$ w \in [ r ] $
	with
	$ p(S_v) = q(S_v) = S_w $.
	\item[(ii)] The partition
	$ p \vee q $
	is non-crossing.
	\item[(iii)] If 
	$ (i_{a+1}, \dots, i_{ b +1}) \in \mathcal{A}_N^{(p,q)} (S) $,
	then
	$ i_{a+1} = i_{ b+1} $.
\end{itemize}
	
\end{lemma}

\begin{proof}
For (i), Corollary \ref{cor:sup} gives the existence of 
$ w\in [r ] $
such that 
$ p(S_v), q(S_v) \subseteq S_w $. 
If
$ \big( \sigma_{w, N} \big)_N = \big( \textrm{Id}_N \big)_N $, 
then 
$ S_w $ 
is a singleton and the statement is trivial. If
$ \big( \sigma_{w, N} \big)_N $
satisfies condition (\ref{max}), then the property follows from Corollary \ref{cor:sup}.

To show (ii), suppose that $ t_1< t_2 < t_3 < t_4 $ 
is a crossing of 
$ p \vee q $, 
i.e.
 $ t_1, t_3 \in B_1 $ 
and 
$ t_2, t_4 \in B_2 $
for 
$ B_1, B_2 $ 
distinct blocks of 
$ p \vee q $.
For 
$ 1 \leq \nu \leq 4 $,
let
 $ v_\nu \in [ r ] $ 
be such that 
$ i_\nu \in S_{v_\nu } $.
From Corollary \ref{cor:pq:2}(iii), we have that 
$ v_1 \neq v_3 $ 
and
$ v_2 \neq v_4 $.
If 
$ v_1 = v_2 $, 
then
$ S_{v_1} $
is not a singleton, so 
$\big( \sigma_{v_4, N} \big)_N \neq \big( \textrm{Id}_N \big)_N $
and part (i) gives that 
$ p(S_{v_1}) = q(S_{v_1}) = S_{v_3} = S_{v_4} $,
which contradicts Lemma \ref{lemma:1-1}.

For (iii), via a circular permutation of the sets 
$ S_1, S_2, \dots, S_r $,  without loss of generality we can suppose that
$ x =1 $, 
that is 
$ a = 0 $
and 
$ S = [ b ] $.	
We shall prove the result by induction on
$ b $.	
	
If
$ b =2 $, 
then
$ S_1 = \{ 1\}$,
$ S_2 = \{ 2 \}$
and
$ p(1) = q(1) = 2 $.
Hence
$ (i_1, i_2, i_3) \in \mathcal{A}_N^{(p, q)}(S) $ 
implies that
	\begin{align*}
	1 =
	\delta_{k_1}^{k_2} \delta_{l_1}^{l_2} 
	=
	\delta_{\sigma_{1, N} (i_1, i_2)}^{\sigma_{2, N}(i_3, i_2)}.
	\end{align*}
If
$ \sigma_{1, N} \neq \sigma_{2, N} $
then condition (\ref{2:free}) gives that 
$\big|  \mathcal{A}_N^{(p, q)} ( \{1, 2\} ) \big|
 = o(N^2) $
and, according to  Corollary \ref{cor:pq:2}(ii), we have 
$ \displaystyle  \lim_{N \rightarrow\infty}	
 \mathcal{E}_N(p, q) = 0 $.
 Therefore
$ \sigma_{1, N} = \sigma_{2, N} $, 
which gives
$ i_1 = i_3 $, 
that is (iii).

To prove the induction step, remember that  the blocks of 
$ p \vee q $
 have an even number of elements and, since
$ p \vee q $ 
is non-crossing, its restriction to
$ [ b ] $
has at least one block
 with consecutive elements,
 $ B = \{ t + 1, t+ 2, \dots, t + s \} $. 
 From Corollary \ref{cor:pq:2}(iii), there is some 
 $ w \in [ r ] $
 such that
 $ t+ \nu \in S_{ w  +\nu } $
 for
 $ \nu = 1, 2, \dots, s $.
 
 Note that the set 
 $ S^\prime =
  S_{ w +1} \cup S_{w +2} \cup \dots \cup S_{w +s} $
  is also invariant under
  $ p $ 
  and 
  $ q $.
  Indeed, if all 
  $S_{w +\nu } $ 
  are singletons, then 
  $ S ^\prime = B $.
  If
   $ S_{w +\nu }$ 
   is not a singleton for some
   $ \nu \in [ s ] $,
   then
   $ \big( \sigma_{ w+\nu, N } \big)_N \neq \big( \textrm{Id}_N \big)_N $,
   and part (i) of the Lemma gives that 
   $ s = 2 $
   and 
   $ p(S_{w +1}) = q(S_{w +1}) = S_{w+2} $.
   It suffices therefore to show that 
$ S^\prime $
 satisfies property (iii); then the conclusion follows by the induction hypothesis applied to 
the set obtained by removing 
$ S^\prime $
from
$ S $.

If 
$ s =2 $, 
then 
$ S^\prime = S_{w+1} \cup S_{w +2} $.
Denote by 
$ n $  the common value of 
$ 
n_{w +1} $ 
and 
$ n_{w +2} $
and
by
$ \alpha +1 $ 
the first element of 
$ S_{w+1} $.
From Lemma \ref{lemma:1-1}, 
$ p(\nu) = q(\nu) = \alpha+2n + 2 - \nu $
for all
$ \nu =1, \dots, n $. 
Therefore, using the construction of 
$(k_\nu, l_{\nu}) $ 
in Section \ref{notations}, we get
\begin{align*}
\big| \mathcal{A}_N^{( p, q)} (S^\prime & ) \big| 
\leq
\big| \{ 
(i_{\alpha +1}, \dots, i_{\alpha+ 2n + 1 }) 
\in [ N ]^{ 2n + 1}
:\  
( k_{\alpha + \nu} , l_{\alpha + \nu } ) = ( k_{ p(\alpha + \nu)}, l_{ q(\alpha + \nu)} ) ,  1 \leq \nu \leq 2n \}  \big| \\
\leq & \big|
\{
(i_{\alpha +1}, \dots, i_{\alpha+ 2n + 1 }) 
\in [ N ]^{ 2n + 1}
:\ 
\sigma_{w + 1, N} ( i_{\alpha +1}, i_{\alpha + n + 1}) = \sigma_{w + 2, N}( i_{\alpha + 2n + 1}, i_{ \alpha + n +1} ) \\
& \hspace{5.5cm}   \textrm{ and }
i_{ \alpha + \nu } = i_{\alpha +  2n + 2 - \nu } \textrm{ whenever } 
2 \leq \nu \leq n
\}\\
\leq & N^{ n - 2 }
\cdot
\big| \{ (i_{\alpha + 1},
 i_{\alpha + n +1 },
 i_{\alpha + 2n + 1}) 
:\
\sigma_{w +1, N} ( i_{\alpha + 1}, i_{\alpha + n +1} ) = 
\sigma_{w+ 2,N} ( i_{\alpha +  2n + 1 }, i_{\alpha+ n + 1} )
\}
\big|.
\end{align*}

If 
$ \big( \sigma_{w+1, N} \big)_N \neq \big( \sigma_{w+2, N}\big)_N $
or if 
$ i_{\alpha+ 1} \neq i_{\alpha + 2n + 1 } $,
then condition (\ref{2:free})
gives that 
\begin{align*}
\big| \{ (i_{\alpha + 1},
i_{\alpha + n +1 },
i_{\alpha + 2n + 1}) 
:\
\sigma_{w +1, N} ( i_{\alpha + 1}, i_{\alpha + n +1} ) = 
\sigma_{w+ 2,N} ( i_{\alpha +  2n + 1 }, i_{\alpha+ n + 1} )
\}
\big| = o(N^2).
\end{align*}
 On the other hand, according to Lemma \ref{lemma:1-1},
$ B(S_{w+1}: S_{w+2})  = n $, 
so
\begin{align*}
\big| \mathcal{A}_N^{(p, q)} (S_{w+1} \cup S_{w+2})\big|
 = 
o(N^{n+1}) = o(N^{1 + | S_{w+1} \cup S_{w+2}| - B(S_{w+1}:S_{w+2})} )
\end{align*}
and
$ \displaystyle \lim_{N \rightarrow \infty } \mathcal{E}_N (p, q) = 0 $
according to Corollary \ref{cor:pq:2}(ii). Therefore
$ \big( \sigma_{w+1, N} \big)_N = \big( \sigma_{w+2, N}\big)_N $
and
$ i_{\alpha+ 1} = i_{\alpha + 2n + 1 } $,
in particular
$ S^\prime $ 
satisfies (iii).

If 
$ s > 2 $, 
then Corollary \ref{cor:sup} gives that
$ \big( \sigma_{ w + \nu ,  N } \big)_N 
= 
\big( \textrm{Id}_N \big)_N $,
 so each 
$ S_{w + \nu} $ 
is a  singleton, that is 
$ S_{w+\nu} = \{ t + \nu \} $
 and also
$ \theta_{w+1} \neq \theta_{w+2} \neq \dots \neq \theta_{w+s} $. 
   
Suppose first that 
$\theta_{ w + 1} = 1 $. 
Then
$ \theta_{w + 2\nu-1} = 1 $
and
$\theta_{w + 2\nu} = \ast $
for all
$ 1 \leq \nu \leq s/ 2 $.
From Lemma \ref{lemma:structure}, either
$ p(t+1) = t + s $ 
or 
$ q(t+1) = t + s $.
If 
$ q(t+1) = t+s $,
then Lemma \ref{lemma:structure} gives that
$ p(t + 2\nu -1 ) = t+2\nu $
for all
$ 1 \leq \nu \leq s/ 2 $,
while
$ q(t+1) = a + s $
and
$ q( t +2\nu) = t + 2\nu + 1 $, 
for all
$ 1 \leq \nu \leq s/2 -1$.
That is 
$ i_{t+s} = i_{t +2\nu} $
and
$ i_{t+1}= i_{ t+2\nu+1} $
for all
$ 1 \leq \nu \leq s/2 $.
It follows that
\begin{align*}
\big|\mathcal{A}_N^{(p, q)}
 (S_{w +1} \cup \dots \cup S_{w +s}) \big|
\leq N^2
\end{align*}
and  
$ \displaystyle \lim_{N \rightarrow \infty } \mathcal{E}_N (p, q) = 0 $
according to Corollary \ref{cor:pq:2}(ii).
If
$ p(t+1) = t + s $, 
then the conditions
$ (k_{t + \nu}, l_{ t + \nu}) 
= ( k_{ p( t + \nu)}, l_{ q( t + \nu)}) $ 
for
$ 1 \leq \nu \leq s $
are equivalent to
$ i_{ t +1}= i_{ t + s + 1} $ 
and
$ i_{t + 2\nu}= i_{ t + 2 \nu +1} $ 
for
$ 1 \leq \nu \leq s/2-1 $
so property (iii) follows.

The case 
$ \theta_{w+1} = \ast $
is similar.
If 
$ p( t +1) = t + s $, 
then this time Lemma \ref{lemma:structure} gives that
$ q(t + 2\nu -1 ) = a+2\nu $
for all
$ 1 \leq \nu \leq s/ 2 $,
while
$ p ( t +2 \nu ) = t + 2 \nu + 1 $, 
for all
$ 1 \leq \nu \leq s/2 -1$.
That is 
$ i_{t+s} = i_{t +2\nu} $
and
$ i_{t+1}= i_{ a+2\nu+1} $
for all
$ 1 \leq \nu \leq s/2 $,
so again
$ \displaystyle 
\lim_{N \rightarrow \infty } \mathcal{E}_N (p, q) = 0 $.

If
$ q(t+1) = t +s $, 
then again Lemma \ref{lemma:structure} gives that
$ p( a + 2 \nu -1 ) = a + 2 \nu $
for all
$ 1 \leq \nu \leq s/2 $
while
$ q ( t + 2 \nu ) = t + 2 \nu + 1 $,
for all
$ 1 \leq \nu \leq s/2 - 1 $.
That is 
$ i_{ t +1} = i_{ t + s + 1} $
and 
$ i_{t + 2\nu}= i_{ t + 2 \nu +1} $ 
for
$ 1 \leq \nu \leq s/2-1 $,
so again (iii) follows.

\end{proof}

\begin{proof}[Proof of Theorem \ref{main}:]
 We shall show that all mixed free cumulants in 
 $ U_N $,  $ U_N^\ast $,
 and elements of the set 
 $ \big\{
  ( U_N^l )^{ \tau_{k, N}},
   \big[( U_N^l )^{ \tau_{k, N}} \big]^\ast 
    : \  1 \leq l \leq M, \
 1 \leq k \leq P \big\}$
  vanish as
 $ N \rightarrow \infty $ 
 and that, for each 
 $ k \in [ P ] $ 
 and each
  $ l \in [ M ] $,
  \begin{align*}
 \lim_{ N  \rightarrow \infty} \kappa_2
  \big( (U_N^l)^{ \tau_{ k, N } }, 
  [(U_N^l)^{ \tau_{ k, N } } ]^\ast    \big)
 = \lim_{ N \rightarrow \infty } \kappa_2 
 \big( [(U_N^l)^{ \tau_{ k, N } } ]^\ast , 
 (U_N^l)^{ \tau_{ k, N } }
 \big) = 1
  \end{align*} 
  and all other free cumulants of
  $ (U_N^l)^{ \tau_{ k, N }} $ 
and
$ [(U_N^l)^{ \tau_{ k, N } } ]^\ast $
vanish as 
$ N \rightarrow \infty $.
 	
 Denote
 $ V_{k N} = \big[ (U_N^{n_k})^{\sigma_{k, N}}  \big]^{\theta_k}$,
 and, without loss of generality (via a taking adjoints and doing a circular permutation of 
 $ S_1, S_2, \dots, S_r $)  assume that 
 $ \theta_1 = 1 $, 
 $ \big( \sigma_{ 1, N } \big)_N \neq \big( \textrm{Id}_N \big)_N $.
 It suffices then to show that
 \begin{align} \label{m-fc-2}
 \lim_{N \rightarrow \infty }  
 \E \circ \tr &
\left(
 V_{1, N} \cdot V_{2, N} \cdots V_{r, N} \right)\\
 =
& \sum_{ t =2}^r \phi(t) \cdot 
\lim_{ N \rightarrow \infty}
 \E \circ \tr
\left(
V_{2, N}
\cdots
V_{t-1, N}
\right) 
\cdot
\lim_{N \rightarrow \infty }
 \E \circ \tr
\left(
V_{t +1, N}
\cdots
V_{r, N} 
\right), \nonumber
 \end{align}
where
 \begin{align*}
 \phi(t)  = \left\{ 
 \begin{array}{ll}
 1, \textrm{ if } V_{t, N}^\ast = V_{1, N}, \textrm{ that is } n_t = n_1, \sigma_{t, N} = \sigma_{1, N}, \theta_t \neq \theta_1\\
 0, \textrm{ otherwise. }
 \end{array}
 \right.
\end{align*}

 To proceed with the proof, note that, with the notations from \ref{notations}, the left-hand side of equation
  (\ref{m-fc-2}) develops as follows:
 \begin{align} 
 \lim_{ N \rightarrow \infty } \label{E:t}
 \E \circ \tr  
 \left(
 V_{1, N} \cdot  V_{2, N} \cdots V_{r, N} \right)
 = & 
 \lim_{N \rightarrow \infty }
 \sum_{ t = 2 }^r
  \sum_{\substack{p,q\in \mathcal{P}_2^\varepsilon(m)\\ p(1) \in S_t }}
   \mathcal{E}_N  ( p, q ).
 \end{align}
 
Let
 $ p, q \in P_2^{\varepsilon} ( m ) $
be are such that 
 $ \displaystyle \lim_{ N \rightarrow \infty} 
 \mathcal{E}_N (p , q ) \neq 0 $
and let
$ x $
be given by
$ p(1) \in S_x $.
From Lemma \ref{lemma:aux}(i), it follows that 
$ \theta_x = \ast $ 
and
$ n_1 $
and 
$ n_x $
have the same value, that we shall denote by
$ n $.
Moreover, Lemma \ref{lemma:1-1} gives that, for all 
$ \nu \in [ n ] $.
\begin{align}\label{nu:2}
p(\nu) = q ( \nu ) = m(x) + 2 - \nu.
\end{align}
So, using that 
$ p \vee q $ 
is non-crossing, as seen in Lemma \ref{lemma:aux}(ii),
we have that 
$ S_2 \cup S_3 \cup \dots \cup S_{x-1} $
and 
$ S_{x+1} \cup \dots S_r $
are invariant under
$ p $
and
$ q $,
hence Lemma \ref{lemma:aux}(iii) gives that 
$ i_{ n+1} = i_{m(x -1) +1} $
and
$ i_{1} = i_{ m(x) + 1} $.
Using the construction of 
$(k_\nu, l_{\nu}) $ 
in Section \ref{notations}, we then get
\begin{align*}
\big| \mathcal{A}^{ ( p, q)}_N & (S_1 \cup S_x)  \big| 
\leq
\big| \{ 
(i_1, i_2, \dots, i_{n+1},
 i_{m(x-1) +1}, i_{ m(x-1) +2}\dots, i_{m(x)+1} )
 \in [ N]^{ 2n +2 }
:\  \\
 &   \hspace{7.cm } ( k_{\alpha + \nu} , l_{\alpha + \nu } ) = ( k_{ p(\alpha + \nu)}, l_{ q(\alpha + \nu)} ) ,  1 \leq \nu \leq n \}  \big| \\
\leq & \big|
\{
(i_1, i_2, \dots, i_{n+1},
i_{m(x-1) +1}, i_{ m(x-1) +2}\dots, i_{m(x)+1} )
\in [ N]^{ 2n +2 } 
:\   \\
& \hspace{ 0.8cm}
\sigma_{1, N} ( i_{1}, i_{ n + 1}) = \sigma_{x, N}( i_{m(x-1) +1 }, i_{ m(x) +1} )  \textrm{ and }
i_{  \nu } = i_{ m(x) + 2 - \nu } \textrm{ whenever } 
1 \leq \nu \leq n
\}\\
\leq & N^{ n - 2 }
\cdot
\big| \{ (i_{1},
i_{ n +1 },
i_{m(x) + 1}) 
:\
\sigma_{1, N} ( i_{1}, i_{ n +1} ) = 
\sigma_{ x,N} ( i_{ 1 }, i_{n + 1} )
\}
\big|.
\end{align*}
But
$ B(S_1: S_x) = n $,
so, using Corollary \ref{cor:pq:2}(ii) and condition (\ref{2:free}), we have  that 
$\displaystyle \lim_{N \rightarrow \infty}
 \mathcal{E}_N ( p, q) = 0 $
 unless
 $ \big( \sigma_{1, N} \big)_N 
 = \big( \sigma_{x, N} \big)_N $.
 We conclude that the summands in the right-hand side of 
equation (\ref{E:t}) vanish as
$ N \rightarrow \infty $
unless
$ \phi(t) = 1 $,
 that is
\begin{align}\label{E:t:2}
 \lim_{ N \rightarrow \infty } 
\E \circ \tr  
\left(
V_{1, N} \cdot  V_{2, N} \cdots V_{r, N} \right)
=
 \lim_{N \rightarrow \infty }
\sum_{ t = 2 }^r
\phi(t) \cdot  \big[
\sum_{\substack{p,q\in \mathcal{P}_2^\varepsilon(m)\\ p(1)\in S_t }}
\mathcal{E}_N  ( p, q )
\big].
\end{align}
 
Moreover, let us define 
$ \varepsilon^\prime : [ m(x-1) - m(1)] \rightarrow \{ 1, \ast\} $
and
$ \varepsilon^{\prime \prime}: [ m - m(x+1)] \rightarrow 
\{ 1, \ast\} $
given by
$ \varepsilon^\prime (\nu) = \epsilon ( \nu + m(1) ) $
and
$ \varepsilon^{\prime\prime} ( \nu )
 = \epsilon ( m(x+1) + \nu )$.
For
$ p_1 \in P_2^{ \varepsilon^\prime} ( m(x-1) - m(1) ) $
and
$ p_2 \in P_2^{ \varepsilon^{ \prime \prime}} ( m - m(x)) $
we define
$ p_1 \diamond p_2 \in P_2^{ \varepsilon} (m ) $
given by
\begin{align*}
p_1\diamond p_2 (\nu) = \left\{ 
\begin{array}{ll}
m(x) + 2 - \nu & \ \textrm{ if } \nu \in S_1 \cup S_x \\
m( 1) + p_1 ( \nu) & \  \textrm{ if } \nu \in S_2 \cup \dots \cup S_{x-1}   \\
m(x) + p_2(\nu)  & \ \textrm{ if } \nu \in S_{x+1} \cup \dots \cup S_r.
\end{array}
\right.
\end{align*}

As seen above, if 
$ p(1) \in S_x $
then
$ \displaystyle \lim_{ N \rightarrow \infty } \mathcal{E}_N ( p, q)  = 0 $
unless
$ p = p_1 \diamond p_2 $
and
$ q = q_1 \diamond q_2 $
for some 
$ p_1, q_1 \in P_2^{ \varepsilon^\prime} ( m(x-1) - m(1)) $
and
$ p_2, q_2 \in P_2^{ \varepsilon^{\prime \prime}}
( m - m(x)) $. 
Therefore
\begin{equation*}
 \lim_{ N \rightarrow \infty } 
\E \circ \tr  
\left(
V_{1, N} \cdot  V_{2, N} \cdots V_{r, N} \right)
=
\lim_{N \rightarrow \infty }
\sum_{ x = 2 }^r
\phi(x) \cdot  \big[
\sum_{\substack{p_1,q_1\in \mathcal{P}_2^{\varepsilon^\prime} (m(x-1) - m(1))\\
p_2, q_2 \in P_2^{ \varepsilon^{ \prime \prime} }
( m - m(x))}}
\mathcal{E}_N  ( p_1 \diamond p_2, q_1 \diamond q_2 )
\big].
\end{equation*}

On the other hand,
\begin{align*}
\E \circ \tr
\left(
V_{2, N}
\cdots
V_{x-1, N}
\right)
 & = 
 \sum_{
 	p_1, q_1 \in P_2^{ \varepsilon^{ \prime}} (m(x-1)- m(1)) } 
 \mathcal{E}_N ( p_1, q_1)
 \\
\E \circ \tr
\left(
V_{x +1, N}
\cdots
V_{r, N} 
\right)
& = 
\sum_{
 p_2, q_2 \in P_2^{ \varepsilon^{ \prime \prime}} (m- m(x)) } 
\mathcal{E}_N ( p_2, q_2),
\end{align*}
so it suffices to show that 
 \begin{align}\label{last}
 \lim_{N \rightarrow \infty }
 \big[ 
 \mathcal{E}_N ( p_1 \diamond p_2, q_1 \diamond q_2)
 - \mathcal{E}_N ( p_1, q_1) 
 \cdot
 \mathcal{E}_N ( p_2, q_2)
 \big]
 = 0,
 \end{align}
 with the convention that 
$ \mathcal{E}( p_1, q_1) = 1 $
if
$ [ m(x-1) ] \setminus [ m(1)] = \emptyset $
and
$\mathcal{E}( p_2, q_2) = 1$
if  
$ [ m ] \setminus [ m (x) ] = \emptyset $.

To prove (\ref{last}), we shall use Lemma \ref{lemma:aux}(iv)  the multiplicative property of the leading term of the Wiengarten function, as mentioned in Preliminaries.

First, note that the restriction of 
$ ( p_1 \diamond p_2 ) \vee ( q_1 \diamond q_2 ) $ 
to 
$ S_1 \cup S_x $
consists on 
$ n $ 
blocks, each with
$ 2 $ 
elements, therefore we have  that
\begin{align*}
 | ( p_1 \diamond p_2 ) \vee ( q_1 \diamond q_2 )| = | p_1 \vee q_1| + | p_2 \vee q_2 | + n 
\end{align*}
and equation (\ref{w-mult}) from Section \ref{weingarten} gives
\begin{align*}
\textrm{Wg}_N (  p_1 \diamond p_2 , & q_1 \diamond q_2 ) \\
 & = \textrm{Wg}_N (p_1, q_1) 
\cdot 
\textrm{Wg}_N (p_2, q_2) 
\cdot
N^{ -n } 
+ O(N^{ -m +
	 |(  p_1 \diamond p_2 ) \vee ( q_1 \diamond q_2 ) | -2 }).
\end{align*}
with 
$ \textrm{Wg}_N (p_1, q_1) 
\cdot 
\textrm{Wg}_N (p_2, q_2) 
\cdot
N^{ -n } = O(N^{ -m + |(  p_1 \diamond p_2 ) \vee ( q_1 \diamond q_2 ) | }) $.

But 
$ \displaystyle
 \mathcal{E}_N ( p, q) = \frac{1}{N}\textrm{Wg}_N (p, q) \cdot 
\big| \mathcal{A}_N^{ (p, q)} \big| = O(N^0) $, 
therefore to prove (\ref{last}) it suffices to show that
\begin{align}\label{f-f-f-n}
\big| \mathcal{A}_N^{(p_1 \diamond p_2, q_1 \diamond q_2)}\big| 
= 
N^{ n-1} \cdot
\big| \mathcal{A}_N^{(p_1, q_1)} \big|
\cdot 
\big| \mathcal{A}_N^{(p_2, q_2)} \big|.
\end{align}

Remember that 
\begin{align*}
\mathcal{A}_N^{ ( p_1 \diamond p_2,  q_1 \diamond q_2 ) }
&=
\big\{ (i_1, \dots, i_m) :\ (k_\nu, l_{\nu}) = (k_{ p_1 \diamond p_2(\nu)}, l_{ q_1 \diamond q_2( \nu )})
\textrm{ for all } \nu \in [m]
\big\}\\
  =  &
 \big\{ (i_1, \dots, i_m) :\
  (k_\nu, l_{\nu}) 
 =
 \left\{ 
 \begin{array}{ll}
 (k_{ m+2 - \nu}, l_{ m+2 - \nu} ) & \textrm{ if } 
 \nu \in S_1 \cup S_x \\
 ( k_{ p_1(\nu)}, l_{ q_1(\nu) } )  & \textrm{ if }
 \nu \in [ m(x-1)] \setminus [m(1)] \\
  ( k_{ p_2(\nu)}, l_{ q_2(\nu) } )  & \textrm{ if }
 \nu \in [ m] \setminus [m(x)]
 \end{array}
 \right.
 \ \big\}.
\end{align*}

For
$ \nu \in S_1 \cup S_x $, 
the conditions 
$
(k_\nu, l_{\nu}) = (k_{ m+2 - \nu}, l_{ m+2 - \nu} ) $
are equivalent to 
$ i_{\nu} = i_{ m(x)+2 - \nu } $
for
$ 2 \leq \nu \leq m(1) $
and
$ \sigma_{ 1, N} ( i_1, i_{m(1) +1}) = 
\sigma_{x, N} ( i_{ m(x) +1}, i_{m(x-1) + 1}) $,
but, as shown above,
$ \sigma_{ 1, N} = \sigma_{x, N} $,
so 
$ i_{\nu} = i_{ m(x) +2 -\nu} $
for all $ \nu \in [ m ( 1 ) + 1 ] $.

On the other hand, the definition of 
$ \mathcal{A}_N^{(p, q)} $
gives that the conditions
\begin{align*}
\left\{
\begin{array}{ll}
(k_\nu, l_{\nu}) = (k_{ p_1 (\nu)}, l_{ q_1 ( \nu )})  & \textrm{ for all } \nu \in [ m(x -1)] \setminus [ m (1) ] \\
i_{ m(x - 1) + 1} = i_{ m(x) +1}
\end{array}
 \right.
\end{align*}
are equivalent to 
$ (i_{m(x-1)}+1, i_{m(x-1)} + 2, \dots, i_{m(x)}) 
\in \mathcal{A}_N^{(p_1, q_1)} $. 
Similarly, the conditions
\begin{align*}
\left\{
\begin{array}{ll}
(k_\nu, l_{\nu}) = (k_{ p_2 (\nu)}, l_{ q_2 ( \nu )})  & \textrm{ for all } \nu \in [ m] \setminus [ m (x) ] \\
i_{ m(x) + 1} = i_{1}
\end{array}
\right.
\end{align*}
are equivalent to 
$ (i_{m(x)}+1, i_{m(x)} + 2, \dots, i_{m}) 
\in \mathcal{A}_N^{(p_2, q_2)} $.

Therefore
$ (i_1, \dots, i_m) \in \mathcal{A}_N^{( p_1 \diamond p_2, 
	q_1 \diamond q_2 )} $
if and only if the following conditions are satisfied:
\begin{enumerate}
	\item[$\bullet$]
	 $ (i_{m(1) +1}, i_{ m(1) + 2}, \dots, i_{ m(x-1)})
	\in \mathcal{A}_N^{(p_1, q_1)} $
	\item[$\bullet$]
	$ (i_{m(x) +1}, i_{ m(x) + 2}, \dots, i_{ m})
	\in \mathcal{A}_N^{(p_2, q_2)} $
	\item[$\bullet$] 
	$ i_1 = i_{ m(x) + 1} \textrm{ and } 
	i_{ m (x-1) + 1 } = i_{ m(1) + 1 } $
	\item[$\bullet$]
	$ i_\nu = i_{ m +2 - \nu } \textrm{ whenever }
	2 \leq \nu \leq m(1) =n. $
\end{enumerate}
%
It follows that 
$
\big| 
\mathcal{A}_N^{( p_1 \diamond p_2, q_1\diamond q_2 )}
 \big| 
 =  
 \big| 
 \mathcal{A}_N^{( p_1 , q_1  )}
 \big|
 \cdot
\big| 
\mathcal{A}_N^{( p_2 , q_2 )}
\big| 
\cdot 
N^{ n -1},
$
which is (\ref{f-f-f-n}). 
\end{proof}

\section{Particular Cases}

We will apply the results from Section 4 to some particular classes of entry permutations. In particular, it will be shown that the original question from \cite{malizy} has an affirmative answer.

\subsection{Partial transposes.} \label{subsect:4:1} The 
$( b, d) $ -partial transpose 
is an entry permutation of matrices of size
$ bd \times bd $
defined as follows. 
Define first the bijection
$ \varphi_{b, d}: 
[ b \cdot d ]^2 \rightarrow \big( [ b ] \times [ d ] \big)^2 $
given by
$ \varphi_{ b, d } ( i, j) 
= ( \alpha_1, \beta_1, \alpha_2, \beta_2 ) $
whenever
\begin{align*}
(i, j) = \big(
( \alpha_1 - 1 )d + \beta_1, ( \alpha_2 -1 ) d + \beta_2 
\big).
\end{align*}
Then define
$ \rho: \big( [ b ] \times [ d ] \big)^2 
\rightarrow
\big( [ b ] \times [ d ] \big)^2 $
by
\begin{align*}
\rho ( \alpha_1, \beta_1, \alpha_2, \beta_2 ) = 
( \alpha_1, \beta_2, \alpha_2, \beta_1).
\end{align*}
The $ (b, d) $-partial transpose is the permutation 
\begin{align*}
\Gamma_{b, d} =
\varphi_{b, d}^{ -1} \circ \rho \circ \varphi_{ b, d }.
\end{align*}

A more intuitive description of the partial transpose is given by seeing each
$ bd \times bd $
square matrix  as a 
$ b \times b $
matrix whose entries are
$ d \times d $ 
block matrices.
Its 
$ (b, d) $-partial transpose is obtained by taking the matrix transpose of each of the 
$ b^2 $ 
blocks, but keeping the position of the blocks.

We shall consider first one sequence of partial transposes. So let us suppose that
$ \big( b_N \big)_N $
and
$ \big( d_N \big)_N $
are two increasing sequences of positive integers 
such that 
$ \displaystyle  \lim_{N \rightarrow\infty } b_N \cdot d_N  = \infty  $.
For fixed 
$ i = (\alpha_1 -1) d + \beta_1 $ 
and 
$ i = ( \alpha_2 -1) d + \beta _2 $,
we have that
\begin{align*}
\sum_{ \alpha =1, 2 }
| \big\{ \nu \in [ N ] :\  i \in \{ \pi_\alpha\circ \Gamma_{ b, d} (\nu, j),  \pi_\alpha \circ \Gamma_{b, d} ( j, \nu) \} \big\}|
= 2 ( \delta_{\alpha_1}^{\alpha_2} \cdot d + \delta_{\beta_1}^{\beta_2} \cdot b)
\end{align*}
so the sequence 
$ \big( \Gamma_{b_N, d_N } \big)_N $
satisfies condition (\ref{max}) if and only if 
$ \displaystyle \lim_{ N  \rightarrow \infty} b_{N}
= \lim_{ N  \rightarrow \infty} d_N = \infty $.
Moreover, the last relation also implies that the sequences
$ \big( \textrm{Id}_N \big)_N $
and 
$ \big( \Gamma_{b_N, d_N } \big)_N $
satisfy relation (\ref{2:free}), since
\begin{align*}
\big|  \big\{ (i_1, i_2, i_3) : \ \Gamma_{b, d} (i_1, i_2) \in \{ (i_1, i_3), (i_3, i_2) \}
\big\}
\big|  = b d^2 + b^2 d.
\end{align*}
Therefore, Theorem \ref{main} implies the following.

\begin{corollary}
If
 $ \big( b_N \big)_N $
and
$ \big( d_N \big)_N $
are two increasing sequences of positive integers 
such that
$ \displaystyle \lim_{ N  \rightarrow \infty} b_{N}
= \lim_{ N  \rightarrow \infty} d_N = \infty $,
then, for any positive integer 
$ P $,
the family
$ U_{ b_N d_N}^{ \Gamma_{b_N, d_N}} $, 
	$ \big( U_{ b_N d_N}^2 \big)^{ \Gamma_{b_N, d_N} }$,
	 $ \dots $
	$\big( U_{ b_N d_N}^P \big)^{ \Gamma_{b_N, d_N}	} 
	$
	is asymptotically free circular and free from
 $ U_{ b_N d_N } $.	
\end{corollary}

 For the case of several diferent sequences of partial transposes, \cite{wishart2} gives a rather simple way to check condition (\ref{2:free}). More precisely, part of Theorem 3.2 from \cite{wishart2} states the following property:
 
 \textit{ If }
 	$ b \cdot d = b^\prime \cdot d^\prime $, 
\textit{ then the relation }
 	$ \Gamma_{ b, d} (i, j) \in 
 \{ \Gamma_{b^\prime , d^\prime} (i, k) ,
 \Gamma_{b^\prime , d^\prime} (k, j) \}	$
 \textit{ is equivalent to }
 $\Gamma_{ b, d } (i, j) = \Gamma_{b^\prime , d^\prime} (i, j) $.
 
 Furthermore, Theorem 3.2 and Theorem 4.8 from \cite{wishart2} give a more practical manner to check condition (\ref{2:free}), stated below.
 
 \textit{Let }
 	$ \big(b_N \big)_N $,
 	$ \big(d_N \big)_N $,
 	$ \big(b_N^\prime \big)_N $,
 	$ \big(d_N^\prime \big)_N $ 
 	\textit{	be increasing sequences of positive integers such that }
 	$ b_N \cdot d_N = b_N^\prime  \cdot d_N^\prime 	$
 	\textit{ for each }
 	$ N $. 
\textit{ The following relations are then equivalent:}
\begin{enumerate}
	\item[(c.1)]
$ \displaystyle 
\lim_{N \rightarrow \infty} \frac{1}{N^2}
\big| \big\{ (i, j) \in [ b_N \cdot d_N]^2 :\
\Gamma_{ b_N, d_N } (i, j) = \Gamma_{b_N^\prime , d_N^\prime} (i, j)
\big\}\big| = 0 $	
	\item[(c.2)]\textit{ If }
	$ Q_{N} = \min\{ d_{N}, d_N^\prime \} \cdot L_{ N} $ 
	\textit{ is the least common multiple of 
		the positive integers }
	$ d_{N} $ 
	\textit{ and }
	$ d_{N}^\prime $,
	\textit{ then }
	$ \displaystyle \lim_{N \rightarrow \infty }
	L_{N} = \infty $.
\end{enumerate}

\noindent (For example, each two of the following sequences satisfy condition (\ref{2:free}):
$ \big( \Gamma_{N, N^{p-1}} \big)_N $,
$ \big( \Gamma_{N^2, N^{p-2}} \big)_N $,
$ \dots $,
$ \big( \Gamma_{N^{p-1}, N} \big)_N $.)

Henceforth, Theorem \ref{main} gives the following corollary, which extends the results from \cite{partial-transpose-haar}:

\begin{corollary}\label{cor:4:2}
	Suppose that for each 
$ i = 1, 2, \dots, n $
there exist two increasing and unbounded sequences of positive integers
$ \big(b_{i, N } \big)_N $
and
$ \big(d_{i, N} \big)_N $
such that:
\begin{enumerate}
	\item[(i)]$ b_{k, N} \cdot d_{k, N} =
	b_{l, N} \cdot d_{l, N} $
	for each 
	$ k, l \in [ n ] $
	and each $ N $
	\item[(ii)] if 
	$ k \neq l $,
	then the sequences
	$ \big(b_{k, N } \big)_N $,
	$ \big(d_{k, N} \big)_N $,
	and
	$ \big(b_{l, N} \big)_N $,
	$ \big(d_{l, N} \big)_N $
	satisfy condition \emph{(c.2)} from above.
\end{enumerate}

Then the family 
$ \big\{  \big( U_{ b_{i, N} \cdot d_{i, N}}^k 
	\big)^{ \Gamma_{ b_{i, N}, d_{i, N} }
	} :\ 
		i \in [ n], k \in [ P ] 
\big\}$
is asymptotically fee circular and free from 
$ \big( U_{ b_{i, N} \cdot d_{i, N} } \big)_N $.

\end{corollary}

\subsection{Mixing map.} With the notations from Subsection \ref{subsect:4:1}, the ``mixing map'' from \cite{mandolino-zycz-lin} is defined as the permutation 
$ \mu_N : [ N^2 ]\times [ N^2] \rightarrow [N^2] \times [ N^2] $
given by
 \begin{align*}
 \mu_N  = \varphi_{N, N}^{ -1} \circ \psi_N \circ \varphi_{N, N}
 \end{align*}
 where
$ \varphi_{N, N}: 
[ N^2 ]^2 \rightarrow  [ N ]^4 $ 
is given by
$ \varphi_{ N, N} ( i, j) 
= ( \alpha_1, \beta_1, \alpha_2, \beta_2 ) $
whenever
\begin{align*}
(i, j) = \big(
( \alpha_1 - 1 )N + \beta_1, ( \alpha_2 -1 ) N + \beta_2 
\big),
\end{align*}
and
$ \psi_N: [ N]^4 \rightarrow [ N ]^4 $ 
is given by
$ \psi_N ( \alpha_1, \beta_1, \alpha_2, \beta_2)
= ( \alpha_1, \alpha_2, \beta_1, \beta_2) $.

Fix 
$ \alpha_1, \alpha_2, \beta_1, \beta_2 \in [ N ] $.
For 
$ a = (\alpha_1 -1)N + \beta_1 $
and
$ b = (\alpha_2 -1) N + \beta_2 $,
we have that
\begin{align*}
\big| \big\{
\nu \in [N^2]:\ a \in \{ \pi_k \circ \nu_N (\nu, b), 
\pi_k\circ\nu_N (b, \nu):\ k=1, 2\} 
\big\} \big| 
&= N( \delta_{\alpha_2}^{\beta_2} +
\delta_{\alpha_2}^{\alpha_1} +
\delta_{\beta_2}^{\beta_1} +
\delta_{\beta_2}^{\alpha_1}
)\\
& \leq 4N = o(N^2),
\end{align*}
so 
$ \mu_N $ 
satisfies condition (\ref{max}).

Also, the pair of sequences
$ \big( \mu_N \big)_N $
and
$ \big( \textrm{Id}_{N^2} \big)_{N^2} $
satisfy condition (\ref{2:free}), because
\begin{align*}
\big|  \big\{ & (i_1,  i_2, i_3) \in [ N^2]^3:\ 
\mu_N (i_1, i_2) \in \{ ( i_1, i_3), (i_3, i_2)\}
\big\}
\big|\\
 & = \big|\big\{
 (\alpha_k, \beta_k)_{1\leq k \leq 3} \in [ N ]^6:\
 (\alpha_1, \alpha_2, \beta_1, \beta_2) \in 
 \{ (\alpha_1, \beta_3, \alpha_3, \beta_1),
 (\alpha_3, \beta_2, \alpha_2, \beta_3)
\}
 \big\}
 \big|\\
 &= O(N^3)=o((N^2)^2).
\end{align*}

Henceforth Theorem \ref{main} gives the following result.
\begin{corollary}\label{cor:4:3}
 For any positive integer $ P$, the family
 $ \big\{  \big( U_{N^2}^k \big)^{ \mu_N }:\  1 \leq k \leq P \big\} $
 is asymptotically free circular and free from 
 $ \big(U_{N^2}\big)_N $.
\end{corollary}

\begin{remark}
Suppose that 
$ \big( b_N \big)_N $, $ \big( d_N \big)_N $
are two increasing, unbounded sequences of positive integers such that
$ b_N \cdot d_N = N^2 $ 
for each 
$ N $.
Then the family
$ \big\{  
\big\{  \big( U_{N^2}^k \big)^{ \mu_N },
\big\{  \big( U_{N^2}^k \big)^{ \Gamma_{b_N, d_N} }
:\ 1 \leq k \leq P
\big\}$
is asymptotically free circular and free from 
$ \big(U_{N^2}\big)_N $.

Indeed, from the arguments above it suffices to show that the pair of sequences
$ \big( \mu_N \big)_N $ 
and
$ \big( \Gamma_{ b_N, d_N } \big)_N $
satisfies condition (\ref{2:free}).

For
$ i_1, i_2, i_3 \in [ N^2 ] $ 
and let 
$ \alpha_1, \alpha_2, \beta_1, \beta_2 \in [ N ] $,
$ a_1, a_3 \in [ b_N] $,  
and
$ b_1, b_2 \in [ d_N ] $
be such that
\begin{align*}
i_1 &= (\alpha_1 - 1)N + \beta_1 = (A_1 -1)d_N + B_1\\
i_2 & = ( \alpha_2 -1) N + \beta_2 \\
i_3 &= (A_3 -1)d_N + B_3.
\end{align*}
The equality 
$ \mu_N (i_1, i_2) = \Gamma_{ b_N, d_N } ( i_1, i_3) $
is then equivalent to 
\begin{align*}
(\alpha_1 - 1)N + \alpha_2 &= (A_1 -1) d_N + B_3\\
( \beta_1  - 1)N + \beta_2 & = (A_3 -1) d_N + B_1
\end{align*}
so, for 
$ A_3 $ 
to be integer, we need 
$ \beta_2 \equiv B_1 - ( \beta_1 -1)N   (\mod d_N) $, therefore
\begin{align*}
\big| \big\{ (i_1, i_2, i_3) \in [ N^2]^3: & \  \mu_N (i_1, i_2) = \Gamma_{ b_N, d_N } ( i_1, i_3) \big\} \big| \\
\leq &\big| \big\{ (i_1, \alpha_2, \beta_2) \in [ N^2] \times [ N ]^2 :\  \beta_2 \equiv B_1 - ( \beta_1 -1)N   (\mod d_N) \big\} \big|\\
\leq  & N^3 \cdot ( 1 + \frac{N}{d_N}) = o(N^4).
\end{align*}

Similarly, 
$ \big| \big\{ (i_1, i_2, i_3) \in [ N^2]^3:  \  \mu_N (i_1, i_2) = \Gamma_{ b_N, d_N } ( i_3, i_2) \big\} \big| = o(N^4) $,
so the conclusion follows.
\end{remark}








\bibliographystyle{alpha}

\end{document}